\documentclass[10pt]{amsart}
\usepackage{amsrefs}
\usepackage{amssymb}
\usepackage[all]{xy}

\usepackage{hyperref}

\usepackage{tikz}
\usetikzlibrary{matrix,arrows}

\DeclareMathOperator{\Gal}{Gal}

\DeclareMathOperator{\Img}{Im}

\DeclareMathOperator{\Ker}{Ker}

\DeclareMathOperator{\res}{res}

\DeclareMathOperator{\trg}{trg}

\DeclareMathOperator{\ab}{ab}

\newcommand{\rmH}{\mathrm{H}}
\newcommand{\bfH}{\mathbf{H}}

\DeclareFontFamily{U}{wncy}{}
\DeclareFontShape{U}{wncy}{m}{n}{<->wncyr10}{}
\DeclareSymbolFont{mcy}{U}{wncy}{m}{n}
\DeclareMathSymbol{\Sha}{\mathord}{mcy}{"58}
\DeclareMathSymbol{\sha}{\mathord}{mcy}{"78}

\begin{document}

\newtheorem{thm}{Theorem}[section]
\newtheorem{cor}[thm]{Corollary}
\newtheorem{lem}[thm]{Lemma}
\newtheorem{fact}[thm]{Fact}
\newtheorem{prop}[thm]{Proposition}
\newtheorem{defin}[thm]{Definition}
\newtheorem{exam}[thm]{Example}
\newtheorem{examples}[thm]{Examples}
\newtheorem{rem}[thm]{Remark}
\newtheorem{case}{\sl Case}
\newtheorem{claim}{Claim}
\newtheorem{question}[thm]{Question}
\newtheorem{conj}[thm]{Conjecture}
\newtheorem*{notation}{Notation}
\swapnumbers
\newtheorem{rems}[thm]{Remarks}
\newtheorem*{acknowledgment}{Acknowledgements}
\newtheorem*{thmno}{Theorem}

\newtheorem{questions}[thm]{Questions}
\numberwithin{equation}{section}

\newcommand{\gr}{\mathrm{gr}}
\newcommand{\inv}{^{-1}}
\newcommand{\isom}{\cong}
\newcommand{\dbC}{\mathbb{C}}
\newcommand{\F}{\mathbb{F}}
\newcommand{\dbN}{\mathbb{N}}
\newcommand{\Q}{\mathbb{Q}}
\newcommand{\dbR}{\mathbb{R}}
\newcommand{\dbU}{\mathbb{U}}
\newcommand{\Z}{\mathbb{Z}}
\newcommand{\calG}{\mathcal{G}}
\newcommand{\K}{\mathbb{K}}
\newcommand{\calX}{\mathcal{X}}
\newcommand{\calY}{\mathcal{Y}}


\newcommand{\hac}{\hat c}
\newcommand{\hatheta}{\hat\theta}

\title[1-cyclotomic oriented pro-$p$ groups]{Chasing maximal pro-$p$ Galois groups \\ via 1-cyclotomicity}
\author{Claudio Quadrelli}
\address{Department of Science and High-Tech, University of Insubria, Como, Italy EU}
\email{claudio.quadrelli@uninsubria.it}
\date{\today}


\begin{abstract}
 Let $p$ be a prime.
We prove that certain amalgamated free pro-$p$ products of Demushkin groups with pro-$p$-cyclic amalgam cannot give rise to a 1-cyclotomic oriented pro-$p$ group, and thus do not occur as maximal pro-$p$ Galois groups of fields containing a root of 1 of order $p$.
We show that other cohomological obstructions which are used to detect pro-$p$ groups that are not maximal pro-$p$ Galois groups --- the quadraticity of $\Z/p\Z$-cohomology and the vanishing of Massey products --- fail with the above pro-$p$ groups.
Finally, we prove that the Mina\v c-T\^an pro-$p$ group cannot give rise to a 1-cyclotomic oriented pro-$p$ group, and we conjecture that every 1-cyclotomic oriented pro-$p$ group satisfy the strong $n$-Massey vanishing property for $n>2$.
\end{abstract}



\maketitle

\section{Introduction}
\label{sec:intro}

Let $p$ be a prime number, and let $1+p\Z_p$ denote the pro-$p$ group of principal units of the ring of $p$-adic integers $\Z_p$ --- namely, $1+p\Z_p=\{1+p\lambda\:\mid\:\lambda\in\Z_p\}$.
An {\sl oriented pro-$p$ group} is a pair $(G,\theta)$ consisting of a pro-$p$ group $G$ and a morphism of pro-$p$ groups $\theta\colon G\to1+p\Z_p$, called an {\sl orientation} of $G$ (see \cite{qw:cyc}; oriented pro-$p$ groups were introduced by I.~Efrat in \cite{efrat:small}, with the name ``cyclotomic pro-$p$ pairs'').
An oriented pro-$p$ group $(G,\theta)$ gives rise to the continuous $G$-module $\Z_p(\theta)$, which is equal to $\Z_p$ as an abelian pro-$p$ group, and which is endowed with the continuous $G$-action defined by
\[
 g\cdot \lambda=\theta(g)\cdot \lambda\qquad\text{for all }g\in G\text{ and }\lambda\in\Z_p(\theta).
\]

An oriented pro-$p$ group $(G,\theta)$ is said to be {\sl Kummerian} if the following cohomological condition is satisfied:
for every $n\geq1$ the natural morphism
\begin{equation}\label{eq:kummer intro}
 \rmH^1(G,\Z_p(\theta)/p^n\Z_p(\theta))\longrightarrow\rmH^1(G,\Z/p\Z),
\end{equation}
induced by the epimorphism of continuous $G$-modules $\Z_p(\theta)/p^n\Z_p(\theta)\twoheadrightarrow\Z/p$ is surjective (see \cite{eq:kummer}) --- here we consider $\Z/p$ as a trivial $G$-module.
Moreover, the oriented pro-$p$ group $(G,\theta)$ is said to be {\sl 1-cyclotomic} if the above cohomological condition is satisfied also for every closed subgroup of $G$ --- namely, the natural morphism \eqref{eq:kummer intro} is surjective also with $H$ instead of $G$, and the restriction $\theta\vert_H\colon H\to1+p\Z_p$ instead of $\theta$ for all closed subgroups $H$ of $G$
(in \cite{cq:galfeat,cq:1smoothBK} a 1-cyclotomic oriented pro-$p$ group is called a ``1-smooth'' oriented pro-$p$ group).
This cohomological condition was considered first by J.~Labute, who showed {\sl ante litteram} that for every Demushkin group $G$ there exists precisely one orientation which completes $G$ into a Kummerian oriented pro-$p$ group, namely, the orientation induced by the dualizing module of $G$ (see \cite{labute:demushkin}).

In case of trivial orientations, 1-cyclotomicity translates into a purely group-theoretical statement.
Namely, an oriented pro-$p$ group $(G,\mathbf{1})$ --- where $\mathbf{1}\colon G\to1+p\Z_p$ denotes the orientation which is constantly equal to 1 --- is 1-cyclotomic if, and only if, the abelianization of every closed subgroup of $G$ is a free abelian pro-$p$ group.
Pro-$p$ groups satisfying this group-theoretic condition are called {\sl absolutely torsion-free} pro-$p$ groups, and they were introduced by T.~W\"urfel in \cite{wurfel}.

The main goal of this work is to produce new examples of pro-$p$ groups which no orientations can turn into a 1-cyclotomic oriented pro-$p$ group.

\begin{thm}\label{thm:intro1}
 Let $G$ be a pro-$p$ group with pro-$p$ presentation
  \begin{equation}\label{eq:pres1}
  G =\left\langle\: x,y_1,\ldots,y_{d_1},z_1,\ldots,z_{d_2}\:\mid\: r_1=r_2=1 \:\right\rangle,
\end{equation}
where $d_1,d_2$ are two positive odd integers, and either:
\begin{itemize}
 \item[(1.1.a)]  $d_1+d_2\geq4$ and
 \[\begin{split}
  r_1 &=[x,y_1][y_2,y_3]\cdots[y_{d_1-1},y_{d_1}], \\
  r_2 &=[x,z_1][z_2,z_3]\cdots[z_{d_2-1},z_{d_2}];
\end{split} \]
\item[(1.1.b)] or $p$ is odd and 
\[\begin{split}
r_1 &=y_1^p[y_1,x][y_2,y_3]\cdots[y_{d_1-1},y_{d_1}],\\
r_2 &= z_1^p[z_1,x][z_2,z_3]\cdots[z_{d_2-1},z_{d_2}].   \end{split}\]
\end{itemize}
Then there are no orientations $\theta\colon G\to1+p\Z_p$ such that the oriented pro-$p$ group $(G,\theta)$ is 1-cyclotomic.
\end{thm}

It is worth underlining that the pro-$p$ groups described in Theorem~\ref{thm:intro1} are amalgamated free pro-$p$ products of two Demushkin groups --- the subgroup generated by $x,y_1,\ldots,y_{d_1}$ and the subgroup generated by $x,z_1,\ldots,z_{d_2}$ ---, with pro-$p$-cyclic amalgam, generated by $x$.
Despite Demushkin groups and their free pro-$p$ products are some of the (extremely few) examples of pro-$p$ groups which are known to give rise to 1-cyclotomic oriented pro-$p$ groups, the presence of a pro-$p$-cyclic amalgam is enough to lose 1-cyclotomicity.

Oriented pro-$p$ groups satisfying 1-cyclotomicity have great prominence in Galois theory.
Given a field $\K$, let $\bar\K_s$ and $\K(p)$ denote respectively the separable closure of $\K$, and the compositum of all finite Galois $p$-extensions of $\K$.
The {\sl maximal pro-$p$ Galois group of $\K$}, denoted by $G_{\K}(p)$, is the maximal pro-$p$ quotient of the absolute Galois group $\Gal(\bar\K_s/\K)$ of ${\K}$, and it coincides with the Galois group of the Galois extension $\K(p)/\K$.
Detecting maximal pro-$p$ Galois groups among pro-$p$ groups, are crucial problems in Galois theory.
Already the pursuit of concrete examples of pro-$p$ groups which do not occur as maximal pro-$p$ Galois groups of fields is already considered a very remarkable challenge (see \cite[\S~25.16]{friedjarden}, and, e.g., \cites{BLMS,cem,cq:noGal,sz:raags,BQW}). 

The maximal pro-$p$ Galois group $G_{\K}(p)$ of a field $\K$ containing a root of 1 of order $p$ gives rise to the oriented pro-$p$ group $(G_{\K}(p),\theta_{\K})$, where $$\theta_{\K}\colon G_{\K}(p)\longrightarrow 1+p\Z_p$$ denotes the {\sl pro-$p$ cyclotomic character} (see Example~\ref{ex:Galois} below).
By Kummer theory, the oriented pro-$p$ group $(G_{\K}(p),\theta_{\K})$ is 1-cyclotomic (see \cite[p.~131]{labute:demushkin} and \cite[\S~4]{eq:kummer}) --- in case $p=2$ we need to assume further that $\sqrt{-1}\in\K$.
Therefore, a pro-$p$ group which cannot complete into a 1-cyclotomic oriented pro-$p$ group does not occur as the maximal pro-$p$ group of a field containing a root of 1 of order $p$ --- and hence neither as the absolute Galois group of any field (see, e.g., \cite[Rem.~3.3]{cq:noGal}).
Hence, the following corollary may be deduced directly from Theorem~\ref{thm:intro1}.

\begin{cor}\label{cor:noGal}
A pro-$p$ group $G$ as in Theorem~\ref{thm:intro1} does not occur as the maximal pro-$p$ Galois group of any field containing a root of 1 of order $p$ (and also $\sqrt{-1}$ if $p=2$).
Hence, $G$ does not occur as the absolute Galois group of any field.
\end{cor}

In the recent past, other cohomological properties have been used to study maximal pro-$p$ Galois groups --- and to find examples of pro-$p$ groups which do not occur as maximal pro-$p$ Galois groups.
 By the Norm Residue Theorem --- proved by M.~Rost and V.~Voevodsky, with the contribution by Ch.~Weibel, see \cite{voev,HW:book} --- one knows that if $\K$ is a field containing a root of 1 of order $p$, the {\sl $\Z/p$-cohomology algebra} $\bfH^\bullet(G_{\K}(p),\Z/p\Z)$, endowed with the {\sl cup-product}
 \[
  \textvisiblespace\smallsmile\textvisiblespace\colon \rmH^m(G_{\K}(p),\Z/p\Z)\times\rmH^n(G_{\K}(p),\Z/p\Z)\longrightarrow
  \rmH^{m+n}(G_{\K}(p),\Z/p\Z),
 \]
is {\sl quadratic}, i.e., its ring structure is completely determined by the 1st and the 2nd cohomology groups (see, e.g., \cite[\S~2]{cq:bk}). 
Moreover, it was shown by E.~Matzri that if $\K$ is a field containing a root of 1 of order $p$, then $G_{\K}(p)$ satisfies the 
{\sl triple Massey vanishing property} (see \cite{idoeli} and references therein) --- for an overview on Massey products in Galois cohomology see \cite{mt:massey}. 
These two cohomological properties were used to find examples of pro-$p$ groups which do not occur as maximal pro-$p$ Galois groups of fields containing a root of 1 of order $p$, for example in \cite[\S~8]{cem} and in \cite[\S~7]{mt:massey}.

We prove that the pro-$p$ groups described in Theorems~\ref{thm:intro1} cannot be ruled out as maximal pro-$p$ Galois groups employing the above two cohomological obstructions.

\begin{prop}\label{prop:properties intro}
 Let G be a pro-$p$ group as in Theorem~\ref{thm:intro1}.
\begin{itemize}
\item[(i)] The $\Z/p$-cohomology algebra $\bfH^\bullet(G,\Z/p\Z)$ is quadratic.
\item[(ii)] The pro-$p$ group $G$ satisfies the cyclic $p$-Massey vanishing property --- namely, the $p$-fold Massey product 
$$\langle\underbrace{\alpha,\ldots,\alpha}_{p\text{ times}}\rangle$$ contains 0 for every $\alpha\in\rmH^1(G,\Z/p\Z)$.
\item[(iii.a)] If $G$ is as in {\rm (1.1.a)}, then $G$ satisfies the 3- and the strong 4-Massey vanishing property.
\item[(iii.b)] If $G$ is as in {\rm (1.1.b)} and $p>3$ then $G$ satisfies the 3- and the strong 4-Massey vanishing property.
 \end{itemize}
\end{prop}

(We recall the basic notions on Massey products in Galois cohomology in \S~\ref{ssec:massey} below.)
Hence, Corollary~\ref{cor:noGal} provides brand new examples of pro-$p$ groups which do not occur as maximal pro-$p$ Galois groups of fields containing a root of 1 of order $p$, and as absolute Galois groups.
Moreover, we remark that the relations which define the pro-$p$ groups described in Theorem~\ref{thm:intro1} are rather ``elementary'' --- just elementary commutators of generator times, possibly, the $p$-power of a generator ---, unlike the examples provided in \cite{BLMS,cem,mt:massey,cq:noGal}, where the relations involve higher commutators.

Finally, we focus on the {\sl Mina\v{c}-T\^an pro-$p$ group}, i.e., the pro-$p$ group $G$ with pro-$p$ presentation
 \[ G=\langle\:x_1,\ldots,x_5\:\mid\:[[x_1,x_2],x_3][x_4,x_5]=1\:\rangle.
 \]
In \cite[\S~7]{mt:massey}, J.~Mina\v{c} and N.D.~T\^an showed that $G$ does not satisfy the 3-Massey vanishing property, and thus it does not occur as the maximal pro-$p$ Galois group of any field containing a root of 1 of order $p$.
We prove that $G$ cannot complete into a 1-cyclotomic oriented pro-$p$ group.

\begin{thm}\label{thm:MinacTan G}
 Let $p$ be an odd prime.
 Then there are no orientations turning the Mina\v{c}-T\^an pro-$p$ group into a 1-cyclotomic oriented pro-$p$ group.
\end{thm}

Theorem~\ref{thm:MinacTan G} has been proved independently by I.~Snopce and P.~Zalesski\u{\i} (unpublished).
Theorem~\ref{thm:MinacTan G} provides a negative answer to the question posed in \cite[Rem.~3.7]{qw:cyc} ---  namely, the 
Mina\v{c}-T\^an pro-$p$ group may be ruled out as a maximal pro-$p$ Galois group of a field containing a root of 1 of order $p$ (and thus as an absolute Galois group) in a ``Massey-free'' way.

Altogether, 1-cyclotomicity of oriented pro-$p$ groups provides a rather powerful tool studying maximal pro-$p$ Galois groups, and it succeeds in detecting pro-$p$ groups which are not maximal pro-$p$ Galois groups when other methods fail, as underlined above.
We believe that further investigations in this direction will lead to new obstructions for the realization of pro-$p$ groups as maximal pro-$p$ Galois group.

Actually, Theorem~\ref{thm:MinacTan G}, and the main result in \cite{sz:raags} (see in particular \cite[p.~1907]{sz:raags}), may lead to the suspect that 1-cyclotomicity is a more restrictive condition in comparison with the vanishing of Massey products.
Thus, we formulate the following conjecture.

\begin{conj}\label{conj}
 Let $(G,\theta)$ be an oriented pro-$p$ group, such that $\Img(\theta)\subseteq1+4\Z_2$ if $p=2$.
If $(G,\theta)$ is 1-cyclotomic, then the pro-$p$ group $G$ satisfies the 3-Massey vanishing property;
if moreover $G$ is finitely generated, then $G$ satisfies the strong $n$-Massey vanishing property for every $n\geq 3$. 
\end{conj}

After the publication on the arXiv of an earlier version of this paper, A.~Merkurjev and F.~Scavia proved the first statement of Conjecture~\ref{conj} --- see \cite[Thm.~1.3]{MerScaH90} ---; while, on the other hand, there are 1-cyclotomic oriented pro-$2$ groups $(G,\theta)$ such that $\Img(\theta)\subseteq1+4\Z_2$, where $G$ is not finitely generated and does not satisfy the strong 4-Massey vanishing property --- see \cite[Thm.~1.6]{MerSca1}.
In particular, \cite[Thm.~1.3]{MerScaH90} implies Theorem~\ref{thm:MinacTan G} (see also \cite[Rem.~6.3]{MerScaH90}).


{\small
\begin{acknowledgment} \rm
The author wishes to thank N.D.~T\^an, who pointed out to the author the possible importance
of \cite[Prop.~6]{labute:demushkin}, some years ago; and P.~Guillot, I.~Snopce, P.~Wake, and P.~Zalesski\u{\i} for several inspiring discussions.

\noindent
This paper was deeply influenced also by the discussions occurred during the following events: 
the workshop ``Nilpotent Fundamental Groups'', Jun.~2017, hosted by the Banff International Research Station (Banff, Canada);
the conference ``Group Theory Days in D\"usseldorf'', Jan.~2020, hosted by the University of D\"usseldorf (D\"usseldorf, Germany);
the conference ``New Trends around Profinite Groups 20'', Sept.~2021, hosted by the Centro Internazionale di Ricerca Matematica  (Levico Terme, Italy);
and the conference ``New Trends around Profinite Groups 22'', Sept.~2022, hosted by the Riemann International School of Mathematics (Varese, Italy).
So, the author is gratheful to the organizers and the hosting institutions of these events.

\noindent 
Last, but not least, the author thanks the anonymous referee, for several helpful comments.
\end{acknowledgment}
}


\section{Oriented pro-$p$ groups and cohomology}\label{sec:or}

\subsection{Notation and preliminaries}

 Throughout the paper, every subgroup of a pro-$p$ group is tacitly assumed to be {\sl closed} with respect to the pro-$p$ topology. 
Therefore, sets of generators of pro-$p$ groups, and presentations, are to be intended in the topological sense.

Given a pro-$p$ group $G$, we denote the closed commutator subgroup of $G$ by $G'$ --- namely, $G'$ is the closed normal subgroup generated by commutators $$[h,g]= h^{-1}\cdot h^g=h^{-1}\cdot g^{-1}hg,\qquad g,h\in G.$$ 
The {\sl Frattini subgroup} of $G$ is denoted by $\Phi(G)$ --- namely, $\Phi(G)$ is the closed normal subgroup generated by $G'$ and by $p$-powers $g^p$, $g\in G$ (cf., e.g., \cite[Prop.~1.13]{ddsms}).
A minimal generating set of $G$ gives rise to a basis of the $\Z/p\Z$-vector space $G/\Phi(G)$, and conversely (cf., e.g., \cite[Prop.~1.9]{ddsms}).

Finally, we denote the abelianization $G/G'$ of $G$ by $G^{\ab}$. 
Throughout the paper, we will make use of the following straightforward fact.

\begin{fact}\label{fact:gen Gab}
Let $G$ be a finitely generated pro-$p$ group.
Then a subset $\{x_1,\ldots,x_d\}$ of $G$ is a minimal generating set of $G$ if, and only if, the subset
 $\{x_1G',\ldots,x_dG'\}$ of $G^{\ab}$ is a minimal generating set of the abelian pro-$p$ group $G^{\ab}$.
\end{fact}

\subsection{Oriented pro-$p$ groups}\label{ssec:oriented}

Let $G$ be a pro-$p$ group.
An orientation $\theta\colon G\to1+p\Z_p$ is said to be {\sl torsion-free} if $p$ is odd, or if $p=2$ and $\Img(\theta)\subseteq1+4\Z_2$.
Observe that one may have an oriented pro-$p$ group $(G,\theta)$ where $G$ has non-trivial torsion and $\theta$ torsion-free (e.g., if $G\simeq\Z/p$ and $\Img(\theta)=\{1\}$).

A morphism of oriented pro-$p$ groups $(G_1,\theta_1)\to(G_2,\theta_2)$,
is a homomorphism of pro-$p$ groups $\phi\colon G_1\to G_2$ such that $\theta_1=\theta_2\circ\phi$ (cf. \cite[\S~3, p.~1888]{qw:cyc}).

Within the family of oriented pro-$p$ groups one has the following constructions.
Let $(G,\theta)$ be an oriented pro-$p$ group.
\begin{itemize}
 \item[(a)] If $N$ is a normal subgroup of $G$ contained in $\Ker(\theta)$, one has the oriented pro-$p$ group
$(G/N,\theta_{/N})$,  where $\theta_{/N}\colon G/N\to1+p\Z_p$ is the orientation such that $\theta_{/N}\circ\pi=\theta$, with $\pi\colon G\to G/N$ the canonical projection.
 \item[(b)] If $A$ is an abelian pro-$p$ group (written multiplicatively), one has the oriented pro-$p$ group
$  A\rtimes(G,\theta)=(A\rtimes G,\tilde\theta)$, with action given by $gag\inv=a^{\theta(g)}$ for every $g\in G$, $a\in A$,
 where the orientation $\tilde\theta\colon A\rtimes G\to1+p\Z_p$ is the composition of the canonical projection 
 $A\rtimes G\to G$ with $\theta$.
\end{itemize}


\subsection{Kummerianity and 1-cyclotomicity}\label{ssec:kummer}

Let $(G,\theta)$ be an oriented pro-$p$ group.
Observe that the $G$-action on the $G$-module $\Z_p(\theta)/p\Z_p(\theta)$ is trivial, as $\theta(g)\equiv 1\bmod p$ for all $g\in G$.
Thus, $\Z_p(\theta)/p\Z_p(\theta)$ is isomorphic to $\Z/p$ as a trivial $G$-module.

An oriented pro-$p$ group $(G,\theta)$ comes endowed with the distinguished subgroup
\[
 K_\theta(G)=\left\langle\:{}^{g}h\cdot h^{-\theta(g)}\:\mid\:g\in G,\:h\in\Ker(\theta)\:\right\rangle
\]
(cf. \cite[\S~3]{eq:kummer}). 
The subgroup $K_\theta(G)$ is normal in $G$, and it is contained in both $\Ker(\theta)$ and $\Phi(G)$.
On the other hand, $K_\theta(G)\supseteq\Ker(\theta)'$, so that $\Ker(\theta)/K_\theta(G)$ is an abelian pro-$p$ group.
Moreover, if $\theta$ is a torsion-free orientation, $G/\Ker(\theta)\simeq\Img(\theta)$ is torsion-free,
and thus either trivial or isomorphic to $\Z_p$.
Hence, the epimorphism $G\twoheadrightarrow G/\Ker(\theta)$ splits, and since 
$ghg^{-1}\equiv h^{\theta(g)}\bmod K_\theta(G)$ for every $g\in G$ and $h\in \Ker(\theta)$, one concludes that
\[
 \left(G/K_{\theta}(G),\theta_{/K_{\theta}(G)}\right)\simeq 
 \dfrac{\Ker(\theta)}{K_\theta(G)}\rtimes \left(G/\Ker(\theta),\theta_{/\Ker(\theta)}\right)
\]
(cf., e.g., \cite[\S~2.2, eq.~(2.1)]{qw:bogomolov}).

One has the following result relating the subgroup $K_\theta(G)$ and the surjectivity of the maps \eqref{eq:kummer intro}
(cf. \cite[Thm.~7.1]{eq:kummer}, see also \cite[Prop.~2.6]{qw:bogomolov}).

\begin{prop}\label{prop:kummer}
Let $(G,\theta)$ be an oriented pro-$p$ group with $\theta$ a torsion-free orientation.
The following are equivalent.
\begin{itemize}
 \item[(i)] The natural map 
 \[
  \rmH^1(G,\Z_p(\theta)/p^n\Z_p(\theta))\longrightarrow \rmH^1(G,\Z/p\Z),
 \]
is surjective for every positive integer $n$.
\item[(ii)] The quotient $\Ker(\theta)/K_\theta(G)$ is a free abelian pro-$p$ group.
\end{itemize} 
\end{prop}

If an oriented pro-$p$ group $(G,\theta)$ with torsion-free orientation satisfies the above two equivalent properties, then it is said to be Kummerian.
Moreover, $(G,\theta)$ is said to be 1-cyclotomic if $(H,\theta\vert_H)$ is Kummerian for every subgroup $H\subseteq G$.

\begin{rem}\label{rem:1cyc def}\rm
The original definition of 1-cyclotomic oriented pro-$p$ group requires only that for every {\sl open} subgroup $U$ of $G$,
 the oriented pro-$p$ group $(U,\theta\vert_U)$ is Kummerian (cf. \cite[\S~1]{qw:cyc}).
 By a continuity argument, this is enough to imply that the oriented pro-$p$ group $(H,\theta\vert_H)$ is Kummerian also for every subgroup $H\subseteq G$ (cf. \cite[Cor.~3.2]{qw:cyc}).
\end{rem}

If $(G,\mathbf{1})$ is an oriented pro-$p$ group with $\mathbf{1}\colon G\to1+p\Z_p$ the orientation constantly equal to 1, then $K_{\mathbf{1}}(G)=G'$, and by Proposition~\ref{prop:kummer} $(G,\theta)$ is Kummerian if, and only if, $G/G'=\Ker(\mathbf{1})/K_{\mathbf{1}}(G)$ is a free abelian pro-$p$ group (cf. \cite[Ex.~3.5--(1)]{eq:kummer}).
Hence, $(G,\mathbf{1})$ is 1-cyclotomic if, and only if, $H/H'$ is a free abelian pro-$p$ group for every subgroup $H\subseteq G$, i.e., $G$ is absolutely torsion-free (cf. \cite[Rem.~2.3]{cq:galfeat}).


\subsection{Examples}\label{ssec:ex}

\begin{exam}\label{ex:Galois}\rm
Let $\K$ be a field containing a root of 1 of order $p$, and also $\sqrt{-1}$ if $p=2$.
Then the pro-$p$ cyclotomic character $\theta_{\K}$ of $G_{K}(p)$ --- induced by the action of $G_{\K}(p)$ on the roots of 1 of $p$-power order contained in $\K(p)$ --- has image contained in $1+p\Z_p$.
Observe that $\Img(\theta_{\K})=1+p^f\Z_p$, where $f\in\dbN\cup\{\infty\}$ is maximal such that $\K$ contains a root of 1 of order $p^f$ (if $f=\infty$, we set $p^\infty=0$).
In particular, $\theta_{\K}$ is a torsion-free orientation.
The module $\Z_p(\theta_{\K})$ is called the {\sl 1st Tate twist of $\Z_p$} (cf., e.g., \cite[Def.~7.3.6]{nsw:cohn}).

\noindent 
For the convenience of the reader, here we recall J.~Labute's argument to show that the oriented pro-$p$ group $(G_{\K}(p),\theta_{\K})$ is Kummerian --- and thus also 1-cyclotomic, as every subgroup $H\subseteq G_{\K}(p)$ is the maximal pro-$p$ Galois group of an extension of $\K$, with pro-$p$ cyclotomic character $\theta_{\K}\vert_H$ ---, as it is presented in \cite[p.~131]{labute:demushkin} (where the module $\Z_p(\theta_{\K})$ is denoted by $I=I(\chi')$).
For every $n\geq1$ one has an isomorphism of continuous $G_{\K}(p)$-modules 
$$\Z_p(\theta_{\K})/p^n\Z_p(\theta_\K)\simeq \mu_{p^n}=\left\{\:\zeta\in\K(p)\:\mid\:\zeta^{p^n}=1\:\right\}.$$
Let $\K^\times$ and $\K(p)^\times$ denote the multiplicative groups of units of $\K$ and $\K(p)$ respectively.
By Hilbert 90, the short exact sequence of continuous $G_{\K}(p)$-modules 
\begin{equation}\label{eq:ses mods K}
 \xymatrix{\{1\}\ar[r] &\mu_{p^n} \ar[r] & \K(p)^\times \ar[r]^-{\textvisiblespace^{p^n}} & \K(p)^\times \ar[r] &\{1\} }
\end{equation}
 induces a commutative diagram
\[
 \xymatrix{\K^\times/(\K^\times)^{p^n} \ar[r]\ar@{->>}[d]^-{\pi_n} & H^1(G_{\K}(p),\mu_{p^n})\ar[r]^-{\sim}\ar[d] & H^1\left(G_{\K}(p),\Z_p(\theta_{\K})/p^n\Z_p(\theta_{\K})\right)\ar[d] \\
 \K^\times/(\K^\times)^{p} \ar[r]^-{\sim} & H^1(G_{\K}(p),\mu_{p})\ar[r]^-{\sim} & H^1\left(G_{\K}(p),\Z/p\Z\right) }
\]
where the left-side vertical arrow $\pi_n$ and the central vertical arrow are induced by the $p^{n-1}$-th power map $\textvisiblespace^{p^n}\colon\K(p)^\times\to\K(p)^\times$,
and the right-side vertical arrow is induced by the epimorphism of continuous $G_{\K}(p)$-modules $\Z_p(\theta_{\K})/p^n\Z_p(\theta_{\K})\twoheadrightarrow\Z/p\Z$.
Since the map $\pi_n$ is surjective, also the other vertical arrows are surjective. 
\end{exam}

\begin{exam}\label{ex:free}\rm
 Let $G$ be a free pro-$p$ group. 
Then the oriented pro-$p$ group $(G,\theta)$ is 1-cyclotomic for any orientation $\theta\colon G\to1+p\Z_p$ (cf. \cite[\S~2.2]{qw:cyc}). 
\end{exam}

\begin{exam}\label{exam:demushkin}\rm
Let $G$ be an infinite Demushkin group (cf., e.g., \cite[Def.~3.9.9]{nsw:cohn}).
 By \cite[Thm.~4]{labute:demushkin}, $G$ comes endowed with a canonical orientation $\chi\colon G\to1+p\Z_p$ which is the only one completing $G$ into a 1-cyclotomic oriented pro-$p$ group.
 In particular, if $d=\dim(\rmH^1(G,\Z/p\Z))$ is even (which is always the case if $p\neq 2$), then $G$ has a presentation
 \[
  G=\left\langle\:x_1,\ldots,x_d\:\mid\:x_1^{p^f}[x_1,x_2]\cdots[x_{d-1},x_d]=1\:\right\rangle,
 \]
with $f\geq1$ ($f\geq2$ if $p=2$).
In this case $\chi(x_2)=(1-p^f)^{-1}$ and $\chi(x_i)=1$ for $i\neq2$.
\end{exam}

\begin{exam}\label{ex:thetabelian}\rm
Let $(G,\theta)$ be an oriented pro-$p$ group, with $\theta$ a torsion-free orientation.
The oriented pro-$p$ group $(G,\theta)$ is said to be {\sl $\theta$-abelian} if the subgroup $K_\theta(G)$ is trivial and if $\Ker(\theta)$ is a free abelian pro-$p$ group --- in this case $G$ is a free abelian-by-cyclic pro-$p$ group, i.e., $$G\simeq \Ker(\theta)\rtimes \dfrac{G}{\Ker(\theta)}$$ (cf. \cite[Rem.~ 2.2]{qw:bogomolov}).
In other words, $G$ has a presentation
\[
 G=\left\langle\:x_0,x_i\:\mid\:i\in I,\: x_i^{x_0}=x_i^{\theta(x_0)^{-1}},[x_i,x_j]=1\:\forall\:i,j\in I\:\right\rangle,
\]
for some set of indices $I$, and $\theta(x_i)=1$ for all $i\in I$ (cf. \cite[Prop.~3.4]{cq:bk}).
A $\theta$-abelian oriented pro-$p$ group $(G,\theta)$ is Kummerian by Proposition~\ref{prop:kummer}, as by definition $K_\theta(G)$ is trivial and $\Ker(\theta)$ is a free abelian pro-$p$ group.
Moreover, if $H$ is a subgroup of $G$, then one has
\[
 H\simeq (H\cap\Ker(\theta))\rtimes\frac{H}{\Ker(\theta\vert_H)}
\]
(cf. \cite[Rem.~2.4]{qw:bogomolov}), so that the oriented pro-$p$ group $(H,\theta\vert_H)$ is $\theta\vert_H$-abelian, and thus Kummerian, and consequently $(G,\theta)$ is 1-cyclotomic.
\end{exam}

One has the following result to check whether an oriented pro-$p$ group is Kummerian (cf. \cite[Prop.~2.6, Prop.~3.6]{qw:bogomolov}).

\begin{prop}\label{prop:kummer thetabelian}
Let $(G,\theta)$ be an oriented pro-$p$ group, with $\theta$ a torsion-free orientation.
Then $(G,\theta)$ is Kummerian if, and only if, there exists a normal subgroup $N$ of $G$ such that $N\subseteq \Ker(\theta)\cap\Phi(G)$, and the quotient $(G/N,\theta_{/N})$,
is a $\theta_{/N}$-abelian oriented pro-$p$ group. 
If such a normal subgroup $N$ exists, then $N=K_\theta(G)$.
\end{prop}


\subsection{Kummerianity and 1-cocyles}\label{ssec:hered}
Let $(G,\theta)$ be an oriented pro-$p$ group.
Recall that for $n\in\dbN\cup\{\infty\}$, a {\sl 1-cocycle} $c\colon G\to\Z_p(\theta)/p^n\Z_p(\theta)$ is a continuous map satisfying 
\begin{equation}\label{eq:1cocycle}
 c(gh)=c(g)+{\overline{\theta(g)}}c(h)\qquad\text{for every }g,h\in G,
\end{equation}
where $\overline{\theta(g)}$ denotes the image of $\theta(g)$ modulo $p^n$. 
From \eqref{eq:1cocycle} one deduces
\begin{equation}\label{eq:1cocycle comm}
 c([g,h])=\overline{\theta(gh)^{-1}}\left(c(g)(1-\overline{\theta(h)})-c(h)(1-\overline{\theta(g)})\right).
\end{equation}
For $n\in\dbN\cup\{\infty\}$, every element of $\rmH^1(G,\Z_p(\theta)/p^n\Z_p(\theta))$ is represented by a 1-cocycle $c\colon G\to\Z_p(\theta)/p^n\Z_p(\theta)$.
The following result is due to J.~Labute (cf. \cite[Prop.~6]{labute:demushkin}).

\begin{lem}\label{lem:labute inf}
Let $(G,\theta)$ be a finitely generated oriented pro-$p$ group with torsion-free orientation, and let $\calX=\{x_1,\ldots,x_d\}$ be a minimal generating set of $G$.
The following are equivalent.
\begin{itemize}
 \item[(i)] $(G,\theta)$ is Kummerian.
 \item[(ii)] For all $n\in\dbN\cup\{\infty\}$ and for any sequence $\lambda_1,\ldots,\lambda_d$ of elements of $\Z_p(\theta)/p^n\Z_p(\theta)$ there exists a continuous 1-cocycle $G\to\Z_p(\theta)/p^n\Z_p(\theta)$ satisfying $c(x_i)=\lambda_i$ for all $i=1,\ldots,d$.
\end{itemize}
\end{lem}

\begin{prop}\label{prop:quot}
Let $G$ be a finitely generated pro-$p$ group, and let $(G,\theta)$ be a Kummerian oriented pro-$p$ group with torsion-free orientation.
If $N$ is a normal subgroup of $G$ such that $N\subseteq\Ker(\theta)$ and the restriction map
\[
\mathrm{res}_{G,N}^1\colon \rmH^1(G,\Z/p\Z)\longrightarrow\rmH^1(N,\Z/p\Z)^G                                                                                                                                                                                               \]
 is surjective, then also $(G/N,\theta_{/N})$ is Kummerian.
\end{prop}

In order to prove Proposition~\ref{prop:quot} we need the following fact, whose proof --- rather straightforward --- is left to the reader.

\begin{fact}\label{fact:cocycle}
Let $G$ be a finitely generated pro-$p$ group, and let $(G,\theta)$ be an oriented pro-$p$ group with torsion-free orientation.
\begin{itemize}
 \item[(i)] If $c\colon G\to \Z_p(\theta)/p^n\Z_p(\theta)$ is a continuous 1-cocycle, with $n\in\dbN\cup\{\infty\}$,
 then $c^{-1}(0)\cap\Ker(\theta)$ is a normal subgroup of $G$. 
 \item[(ii)] Let $N\subseteq G$ be a normal subgroup satisfying $N\subseteq\Ker(\theta)$, with canonical projection $\pi\colon G\to G/N$. For $n\in\dbN\cup\{\infty\}$ one has the following:
\begin{itemize}
 \item[(a)] a continuous 1-cocycle $c\colon G\to\Z_p(\theta)/p^n\Z_p(\theta)$ satisfying $c\vert_N\equiv0$ induces a continuous 1-cocycle $\bar c \colon G/N\to\Z_p(\theta_{/N})/p^n\Z_p(\theta_{/N})$ such that $c=\bar c\circ\pi$;
 \item[(b)] a continuous 1-cocycle $\bar c \colon G/N\to\Z_p(\theta_{/N})/p^n\Z_p(\theta_{/N})$ induces a continuous 1-cocycle $c\colon G\to\Z_p(\theta)/p^n\Z_p(\theta)$ satisfying $c\vert_N\equiv0$ and $c=\bar c\circ\pi$.
\end{itemize}
\end{itemize}
\end{fact}

\begin{proof}[Proof of Proposition~\ref{prop:quot}]
Set $\bar G=G/N$ and $\bar\theta=\theta_{/N}$.
 For every $n\geq1$, the canonical projection $\pi\colon G\to\bar G$ induces the inflation maps
\begin{equation}
\begin{split}
 f_n &\colon \rmH^1(\bar G,\Z_p(\bar\theta )/p^n\Z_p(\bar\theta ))\longrightarrow \rmH^1(G,\Z_p(\theta)/p^n\Z_p(\theta)), \\
 f&\colon \rmH^1(\bar G,\Z/p\Z)\longrightarrow \rmH^1(G,\Z/p\Z) ,
\end{split}\end{equation}
 which are injective by \cite[Prop.~1.6.7]{nsw:cohn}.
Also, the epimorphisms (respectively of continuous $\bar G$-modules and continuous $G$-modules) 
$\Z_p(\bar\theta )/p^n\Z_p(\bar\theta )\to\Z/p\Z$ and $\Z_p(\theta)/p^n\to\Z/p\Z$ induce, respectively, the morphisms 
\begin{equation}
\begin{split}
\tau_n^N&\colon \rmH^1(\bar G,\Z_p(\theta)/p^n)\longrightarrow \rmH^1(\bar G,\Z/p),\\
\tau_n&\colon \rmH^1(G,\Z_p(\theta)/p^n)\longrightarrow \rmH^1(G,\Z/p).\end{split}\end{equation}
Altogether, by \cite[Prop.~1.5.2]{nsw:cohn} one has the commutative diagram
 \[
  \xymatrix@C=1.2truecm{ \rmH^1\left(\bar G ,\Z_p(\bar\theta )/p^n\Z_p(\bar\theta )\right)\ar[r]^-{\tau_n^N}\ar[d]^-{f_n} & \rmH^1(\bar G ,\Z/p\Z)\ar[d]^-{f} \\
\rmH^1\left(G,\Z_p(\theta)/p^n\Z_p(\theta)\right)\ar@{->>}[r]^-{\tau_n} & \rmH^1(G,\Z/p\Z) }
 \]
Since $(G,\theta)$ is Kummerian, $\tau_n$ is surjective for every $n\geq1$.
 Given  $\bar \beta\in \rmH^1(\bar G ,\Z/p\Z)$, $\bar\beta\neq0$,
our goal is to find $\alpha\in \rmH^1(\bar G ,\Z_p(\bar\theta )/p^n\Z_p(\bar\theta ))$ such that $\bar\beta=\tau_n^N(\alpha)$.

Set $\beta=\bar\beta\circ\pi=f(\bar\beta)$.
Then $\beta\colon G\to\Z/p$ is a non-trivial continuous homomorphism
such that $\Ker(\beta)\supseteq N$.
By hypothesis, the morphism $ N/N^p[G,N]\to G/\Phi(G)$
induced by the inclusion $N\hookrightarrow G$, and dual to $\mathrm{res}_{G,N}^1$, is injective.
Thus, one may find a minimal generating set $\calX$ of $G$ such that $\calY=\calX\cap N$ generates $N$ as a normal subgroup of $G$.
By Lemma~\ref{lem:labute inf}, there exists a continuous 1-cocycle $c\colon G\to\Z_p(\theta)/p^n\Z_p(\theta)$
satisfying 
$$c(x)\equiv\beta(x)\mod p\Z_p(\theta)\qquad\text{for every }x\in \calX$$ 
--- i.e., $\tau_n([c])=\beta$, where $[c]\in \rmH^1(G,\Z_p(\theta)/p^n\Z_p(\theta))$ denotes the cohomology class of $c$ ---, and moreover $c(x)=0$ for every $x\in \calY$.
Therefore, by Fact~\ref{fact:cocycle}--(i), the restriction 
$$c\vert_N\colon N\longrightarrow\Z_p(\theta)/p^n\Z_p(\theta)$$ is the map constantly equal to 0.
By Fact~\ref{fact:cocycle}--(ii), $c$ induces a continuous 1-cocycle 
$$\bar c\colon \bar G \longrightarrow\Z_p(\bar \theta)/p^n\Z_p(\bar \theta)$$ such that $\bar c\circ\pi=c$,
and $[c]=f_n([\bar c])$, where $[\bar c]\in \rmH^1(\bar G ,\Z_p(\bar\theta)/p^n\Z_p(\bar \theta))$ denotes the cohomology class of $\bar c$.
Altogether, one has 
\[
 f(\bar \beta)=\beta=\tau_n([c])=\tau_n\circ f_n([\bar c])=f\circ\tau_n^N([\bar c]).\]
Since $f$ is injective, one obtains $\bar\beta=\tau_n^N([\bar c])$.
\end{proof}

\begin{rem}\rm
 Proposition~\ref{prop:quot} may be proved also in a purely group-theoretic way, see \cite[Rem.~3.9]{BQW}.
\end{rem}

\section{The $\Z/p\Z$-cohomology of $G$}\label{sec:new}

The purpose of this section is to prove the first statement of Proposition~\ref{prop:properties intro}, and more in general to describe the $\Z/p\Z$-cohomology algebra $\bfH^\bullet(G,\Z/p\Z)$ with $G$ as in Theorem~\ref{thm:intro1}.

\subsection{Degree 1 and 2}

Let $G$ be a pro-$p$ group.
We set the subgroup $G_{(3)}$ of $G$ as follows:
\[
 G_{(3)}=\begin{cases}
          G^p[G,G'] & \text{if }p\neq2,\\ G^4(G')^2[G,G'] & \text{if }p=2,
         \end{cases}\]
i.e., $G_{(3)}$ is the third term of the $p$-Zassenhaus filtration of $G$ (cf., e.g., \cite[\S~3.1]{cq:2relUK}).
In particular, $G_{(3)}$ is a normal subgroup of the Frattini subgroup $\Phi(G)$, and the quotient $\Phi(G)/G_{(3)}$ is a $p$-elementary abelian pro-$p$ group --- and thus also a $\Z/p$-vector space.

Recall that the cohomology group $\rmH^1(G,\Z/p\Z)$ is equal to the group of pro-$p$ group homomorphisms from $G$ to $\Z/p$, namely, one has 
 \begin{equation}\label{eq:H1 dual}
 \rmH^1(G,\Z/p\Z)=\mathrm{Hom}(G,\Z/p\Z)\simeq(G/\Phi(G))^\ast,
\end{equation}
where $\textvisiblespace^\ast$ denotes the $\Z/p$-dual (cf., e.g., \cite[Ch.~I, \S~4.2]{serre:galc}).
Thus, the dimension of $\rmH^1(G,\Z/p\Z)$ is equal to the cardinality $\mathrm{d}(G)$ of any minimal generating set of $G$.
On the other hand, the dimension of $\rmH^2(G,\Z/p\Z)$ is equal to the number $\mathrm{r}(G)$ of defining relations of $G$
(cf. \cite[Ch.~I, \S~4.3]{serre:galc}).
Moreover, if both $\rmH^1(G,\Z/p\Z)$ and $\rmH^2(G,\Z/p\Z)$ are finite, and if the cup-product yields an epimorphism 
$\rmH^1(G,\Z/p\Z)^{\otimes2}\twoheadrightarrow\rmH^2(G,\Z/p\Z)$, one has an isomorphism of elementary abelian $p$-groups
 \begin{equation}\label{eq:H2 dual}
 \xymatrix{  \left(\Phi(G)/G_{(3)}\right)^\ast \ar[r]^-{\trg} & \rmH^2(G,\Z/p\Z)}
\end{equation}
(cf. \cite[Thm.~7.3]{MPQT}).
For further properties of the cohomology of pro-$p$ groups we refer to \cite[Ch.~I, \S~4]{serre:galc} and to \cite[Ch.~III, \S~9]{nsw:cohn}.


\subsection{Amalgams}\label{ssec:amalg}
Henceforth $G$ will denote a pro-$p$ group as in Theorem~\ref{thm:intro1}.
Set
\[
 \begin{split}
  G_1 &= \langle\:x,y_1,\ldots,y_{d_1}\:\mid\:x^{\epsilon p}[x,y_1]\cdots[y_{d_1-1},y_{d_1}]=1\:\rangle,\\
  G_2 &= \langle\:x,z_1,\ldots,z_{d_2}\:\mid\:x^{\epsilon p}[x,z_1]\cdots[z_{d_2-1},z_{d_2}]=1\:\rangle,
 \end{split}
\]
with $\epsilon=0,1$ depending on whether we are considering case (1.1.a) or (1.1.b).
Then $G_1,G_2$ are Demushkin groups, and $G$ is the amalgamated free pro-$p$ product 
\begin{equation}\label{eq:amalg}
 G=G_1\amalg_X^{\hat p}G_2, 
\end{equation}
with amalgam the subgroup $X\subseteq G_1,G_2$ generated by $x$.
Observe that $X\simeq\Z_p$, as $X$ has infinite index in both $G_1,G_2$, and subgroups of infinite index of Demushkin groups are free pro-$p$ groups (cf. \cite[Ch.~I, \S~4.5, Ex.~5--(b)]{serre:galc}).
Therefore, the amalgamated free pro-$p$ product is proper, i.e., $G_1,G_2\subseteq G$ (cf. \cite{ribes:amalg}).


\subsection{Quadratic cohomology}\label{ssec:quadratic}

Let 
$$\mathcal{B}=\left\{\:\chi,\:\varphi_1,\:\ldots,\:\varphi_{d_1},\:\psi_1,\:\ldots,\:\psi_{d_2}\:\right\}$$
be the basis of $\rmH^1(G,\Z/p\Z)=\mathrm{Hom}(G,\Z/p\Z)$ dual to $\calX=\{x,y_1,\ldots,z_{d_2}\}$ --- i.e., 
\[\begin{split}
& \chi(w)=\begin{cases} 1&\text{if }w=x \\ 0 &\text{if }w=y_i,z_j \end{cases}\qquad\text{and}\\
& \varphi_i(w)=\begin{cases} \delta_{i,i'}&\text{if }w=y_{i'} \\ 0 &\text{if }w=x,z_j, \end{cases}\qquad
 \psi_j(w)=\begin{cases} \delta_{j,j'}&\text{if }w=z_{j'} \\ 0 &\text{if }w=x,y_i, \end{cases}
\end{split}\]
for every $1\leq i,i'\leq d_1$ and $1\leq j,j'\leq d_2$ (cf. \eqref{eq:H1 dual}).
With an abuse of notation, we may consider the subsets $\mathcal{B}_1=\{\chi,\varphi_1,\ldots,\varphi_{d_1}\}$,
$\mathcal{B}_2=\{\chi,\psi_1,\ldots,\psi_{d_2}\}$, and $\mathcal{B}_X=\{\chi\}$, as bases of $\rmH^1(G_1,\Z/p\Z)$, $\rmH^1(G_2,\Z/p\Z)$, and $\rmH^1(X,\Z/p\Z)$ respectively.

\begin{prop}\label{prop:quadratic}
 The algebra $\bfH^\bullet(G,\Z/p\Z)$ is quadratic.
\end{prop}

\begin{proof}
 As stated in \S~\ref{ssec:amalg}, $G=G_1\amalg_{X}^{\hat p}G_2$ is a proper amalgamated free pro-$p$ product.
 Since $\mathcal{B}_X\subseteq \mathcal{B}_1,\mathcal{B}_2$, the restriction maps
 $$\res^1_{G_i,X}\colon \rmH^1(G_i,\Z/p\Z)\longrightarrow\rmH^1(X,\Z/p\Z),\qquad \text{with }i=1,2,$$
 are surjective.
 Moreover, $\rmH^2(X,\Z/p\Z)=0$, as $X\simeq\Z_p$, and thus $\Ker(\res_{G_i,X}^2)=\rmH^2(G_i,\Z/p\Z)$ for both $i=1,2$.
 On the other hand, $\rmH^1(G_1,\Z/p\Z)$ and $\rmH^1(G_2,\Z/p\Z)$ are generated by $\chi\smallsmile\varphi_1$ and $\chi\smallsmile\psi_1$ respectively, as $G_1,G_2$ are Demushkin groups (cf., e.g., \cite[Prop.~3.9.16]{nsw:cohn}), and thus 
 \[
  \Ker(\res_{G_i,X}^2)=\rmH^2(G_i,\Z/p\Z)=\Ker(\res_{G_i,X}^1)\smallsmile\rmH^1(G_i,\Z/p\Z),\qquad \text{with }i=1,2,
 \]
as $\res_{G_1,X}^1(\varphi_1)=0$ and $\res_{G_2,X}^1(\psi_1)=0$.
Finally, Demushkin groups are well-known to yield a quadratic $\Z/p\Z$-cohomology algebra, while $\bfH^\bullet(X,\Z/p\Z)$ is obviously quadratic, as $X\simeq \Z_p$.
Therefore, we may apply \cite[Thm.~B]{qsv:quadratic}, so that also $\bfH^\bullet(G,\Z/p\Z)$ is quadratic.
\end{proof}

We describe now more in detail the structure of $\bfH^\bullet(X,\Z/p\Z)$.
By duality --- cf. \cite[Thm.~7.3]{MPQT} and \eqref{eq:H2 dual} ---, the set 
$\{\chi\smallsmile\varphi_1,\chi\smallsmile\psi_1\}$ is a basis of $\rmH^2(G,\Z/p\Z)$, and in $\rmH^2(G,\Z/p\Z)$ one has the relations
\begin{equation}\label{eq:H2 G1 a}
  \chi\smallsmile\varphi_{i'}=\chi\smallsmile\psi_{j'}=\varphi_i\smallsmile\psi_j=0
\end{equation}
for all $1\leq i,i'\leq d_1$ and $1\leq j,j'\leq d_2$, with $i',j'\neq1$, and
\begin{equation}\label{eq:H2 G1 b}
\begin{split}
 \varphi_{i}\smallsmile\varphi_{i'}=\begin{cases} (-1)^\epsilon\chi\smallsmile\varphi_1 & \text{if }2\mid i=i'-1,
\\ 0&\text{otherwise}, \end{cases}\\
\qquad
\psi_{j}\smallsmile\psi_{j'}=\begin{cases} (-1)^\epsilon\chi\smallsmile\psi_1 & \text{if }2\mid j=j'-1, 
\\ 0&\text{otherwise} \end{cases} 
\end{split}\end{equation}
(see also \cite[\S~3.2]{cq:2relUK}).

Finally, one has an exact sequence
\[
  \begin{tikzpicture}[descr/.style={fill=white,inner sep=2pt}]
        \matrix (m) [
            matrix of math nodes,
            row sep=3em,
            column sep=3em,
            text height=1.5ex, text depth=0.25ex
        ]
        {    \cdots &  \rmH^2(X,\Z/p\Z) &\\
             \rmH^3(G,\Z/p\Z) &  \rmH^3(G_1,\Z/p\Z)\oplus\rmH^3(G_2,\Z/p\Z) & \cdots  \\
           };

        \path[overlay,->, font=\scriptsize,>=latex]
        (m-1-1) edge node[auto] {} (m-1-2) 
        (m-1-2) edge [out=355,in=175] node[auto] {} (m-2-1)
        (m-2-1) edge node[auto] {} (m-2-2)
        (m-2-2) edge node[auto] {} (m-2-3);
\end{tikzpicture}
\]
(cf. \cite[p.~653]{qsv:quadratic}).
Since $\rmH^2(X,\Z/p\Z)=\rmH^3(G_i,\Z/p\Z)=0$ for both $i=1,2$, one has $\rmH^3(G,\Z/p\Z)=0$, and thus by quadraticity also $\rmH^n(G,\Z/p\Z)=0$ for all $n\geq3$.

\begin{rem}\label{rem:torfree}\rm
It is well-known that if a pro-$p$ group has non-trivial torsion, then its $n$-th $\Z/p$-cohomology group is non trivial for every $n>0$; hence, $G$ is torsion-free.
\end{rem}


\section{Proof of Theorem~\ref{thm:intro1} case (1.1.a)}

Let $G$ be a pro-$p$ group as defined in Theorem~\ref{thm:intro1}, with defining relations as in (1.1.a) --- namely,
\[
  G=\langle\:x,y_1,\ldots,y_{d_1},z_1,\ldots,z_{d_2}\:\mid\:r_1=r_2=1\:\rangle,\]
with $d_1+d_2\geq4$ and 
\[\begin{split}
 r_1&= [x,y_1]\cdots[y_{d_1-1},y_{d_1}],\\r_2&=[x,z_1]\cdots[z_{d_2-1},z_{d_2}].
  \end{split}\]
Without loss of generality, we may assume that $d_1\geq3$.


\subsection{Kummerianity}

Let $G_1,G_2$ be the two Demushkin groups as in \S~\ref{ssec:amalg}, with $\epsilon=0$.
By Example~\ref{exam:demushkin}, if 
$$\theta_1\colon G_1\longrightarrow1+p\Z_p\qquad \text{and} \qquad \theta_2\colon G_2\longrightarrow1+p\Z_p$$
are two torsion-free orientations completing respectively $G_1$ and $G_2$ into Kummerian oriented pro-$p$ groups, then necessarily $\theta_1(x)=\theta_1(y_1)=\ldots=\theta_1(y_{d_1})=1$, and analogously $\theta_2(x)=\theta_2(z_1)=\ldots=\theta_1(z_{d_2})=1$.

\begin{prop}\label{prop:thm2 kummer}
Let $\theta\colon G\to1+p\Z_p$ be a torsion-free orientation.
Then the oriented pro-$p$ group $(G,\theta)$ is Kummerian if, and only if, $\theta$ is constantly equal to 1. 
\end{prop}

\begin{proof}
 If $\theta\equiv\mathbf{1}$, then $(G,\mathbf{1})$ is Kummerian if, and only if, the abelianization $G^{\ab}$ is a free abelian pro-$p$ group.
 But this is easily verified, as clearly $G^{\ab}\simeq\Z_p^{d_1+d_2-1}$.

Conversely, suppose that $(G,\theta)$ is Kummerian.
Let $N_1$ and $N_2$ denote the normal subgroups of $G$ generated as normal subgroups by $z_1,\ldots,z_{d_2}$ and $y_1,\ldots,y_{d_1}$ respectively.
Then $G/N_1\simeq G_1$ and $G/N_2\simeq G_2$.
Moreover, Proposition~\ref{prop:quot} implies that $(G/N_i,\theta_{/N_i})$ is Kummerian for both $i=1,2$.
Since $G/N_i\simeq G_i$ for both $i$, Example~\ref{exam:demushkin} and the argument before the statement of the proposition imply that the torsion-free orientations $\theta_{/N_1}$ and $\theta_{/N_2}$ are constantly equal to $1$.
Hence, also $\theta$ is constantly equal to $1$, as $\theta(w)=\theta_{/N_1}(wN_1)$ for every $w\in G_1$, and analogously 
$\theta(w)=\theta_{/N_2}(wN_2)$ for every $w\in G_2$.
\end{proof}

Therefore, if $G$ may complete into a 1-cyclotomic oriented pro-$p$ group, then necessarily $G$ is absolutely torsion-free.
In order to prove Theorem~\ref{thm:intro1} in case (1.1.a), we aim at exhibiting an open subgroup $H$ of $G$, of index $p^2$, whose abelianization $H^{\ab}$ has non-trivial torsion.


\subsection{The subgroup $U$}

Set $u=y_3^p$, $t_0=z_1^{-1}y_3$, and $t_h=t_0t_0^{y_3}\cdots t_0^{y_3^h}$ for all $h=0,\ldots,p-1$.
A straightforward computation shows that 
\begin{equation}\label{eq:z1 h}
 z_1^{h}=y_3^h\cdot(t_0^{-1})^{y_3^{h-1}}\cdots(t_0^{-1})^{y_3}\cdot t_0^{-1}=y_3^ht_{h-1}^{-1}  
\end{equation}
for all $h=0,\ldots,p-1$.

Let $\phi_G\colon G\to\Z/p$ be the homomorphism of pro-$p$ groups defined by $\phi_G(y_3)=\phi_G(z_1)=1$ and $\phi_G(x)=\phi_G(y_i)=\phi_G(z_j)=0$ for all $i=1,2,4,\ldots,d_1$ and $j=2,\ldots,d_2$, and set $U=\Ker(\phi)$.
Then $U$ is an open subgroup of $G$ of index $p$, generated as a normal subgroup by the subset
\[
 \calX=\left\{\: u,\: x,\: t_0,\: y_i,\:z_j\:\mid \: i=1,2,4,\ldots,d_1,\:j=2,\ldots,d_{2} \right\},
\]
and $G/U=\{U,y_3U,\ldots,y_3^{p-1}U\}$.

\begin{lem}\label{lem:U ab}
 The subset 
\[
\calY_U=\left\{\:u,\:x,\:y_2,\:t_h,\:y_i^{y_3^h},\:z_j^{y_3^h}\:\mid\:i=1,4,\ldots,d_1,\:j=2,\ldots,d_{2},\:h=0,\ldots,p-1\:\right\}      \] of $U$
is a minimal generating set of $U$ as a pro-$p$ group.
\end{lem}

\begin{proof}
 Since $U$ is normally generated by $\calX$ and $G/U=\{U,\ldots,y_3^{p-1}U\}$, $U$ is generated as a pro-$p$ group by the set $\{w^{y_3^h}\mid w\in\calX,h=0,\ldots,p-1 \}$.
 Also, $U$ is subject to the relations
 \begin{eqnarray}\label{eq:rel1 U}
   r_1^{y_3^h}&=&\left[x^{y_3^h},y_1^{y_3^h}\right]\cdots\left[y_{d_1-1}^{y_3^h},y_{d_1}^{y_3^h}\right]=1,\\
   r_2^{y_3^h}&=&\left[x^{y_3^h},z_1^{y_3^h}\right]\cdots\left[z_{d_2-1}^{y_3^h},z_{d_2}^{y_3^h}\right]=1,  \label{eq:rel2 U}
 \end{eqnarray}
with $h=0,\ldots,p-1$.

Consider the abelianization $U^{\ab}$.
Since the only factor in \eqref{eq:rel1 U} which does not lie in $U'$ is $[y_2^{y_3^h},y_3]$, the relation \eqref{eq:rel1 U} implies that $[y_2^{y_3^h},y_3]\in U'$ as well, and therefore
\[
 y_2^{y_3^h}\equiv y_2\mod U'\qquad \text{for all }h=0,\ldots,p-1.
\]
Analogously, the only factor in \eqref{eq:rel2 U} which does not lie in $U'$ is $[x^{y_3^h},z_1^{y_3^h}]$, so that the relation \eqref{eq:rel1 U} implies that $[x^{y_3^h},z_1^{y_3^h}]\in U'$ as well.
Hence, one has
\[\begin{split}
    [x,z_1]\equiv1\bmod U'\; &\Rightarrow \; x^{y_3t_0^{-1}}\equiv x\bmod U'\\
    &\Rightarrow x^{y_3}\equiv x^{t_0}\bmod U',\\
    \left[x^{y_3},z_1^{y_3}\right]\equiv 1\bmod U' \;&\Rightarrow \; 
    (x^{y_3})^{(z_1^{y_3})}=x^{y_3^2(t_0^{-1})^{y_3}}\equiv x^{y_3}\bmod U'\\
    &\Rightarrow\; x^{y_3^2}\equiv x^{t_1}\bmod U',
  \end{split}\]
and so on.
Thus
\[
 x^{y_3^h}\equiv x^{t_{h-1}}\mod U'\qquad \text{for all }h=1,\ldots,p-1.
\]
Altogether, $U^{\ab}$ is the free abelian pro-$p$ group generated by the cosets $\{wU'\mid w\in\calY_U\}$, so that Fact~\ref{fact:gen Gab} yields the claim.
\end{proof}

Now set $U_1=G_1\cap U$ and $U_2=G_2\cap U$.
Then $U_1,U_2$ are open subgroups of $G_1,G_2$ respectively of index $p$, and thus they are again Demushkin groups, on $2+p(d_1-1)$ and $2+p(d_2-1)$ generators respectively (cf. \cite{dummitlabute}).
In particular, the defining relation of $U_1$ is 
\begin{equation}\label{eq:rel U1}
 s_1= \prod_{h=p-1}^0\left(\left[y_{4}^{y_3^h},y_{5}^{y_3^h}\right]\cdots
 \left[y_{d_1-1}^{y_3^h},y_{d_1}^{y_3^h}\right]\left[x^{y_3^h},y_{1}^{y_3^h}\right]\right)[y_2,u]=1,
\end{equation}
while the defining relation of $U_2$ is
\begin{equation}\label{eq:rel U2}
\begin{split}
 s_2&=\prod_{h=p-1}^0\left(\left[z_{2}^{z_1^h},z_{3}^{z_1^h}\right]\cdots
 \left[z_{d_2-1}^{z_1^h},z_{d_2}^{z_1^h}\right]\right)[x,z_1^p]  \\
 &=\prod_{h=p-1}^0\left(\left[z_{2}^{y_3^ht_{h-1}^{-1}},z_{3}^{y_3^ht_{h-1}^{-1}}\right]\cdots
 \left[z_{d_2-1}^{y_3^ht_{h-1}^{-1}},z_{d_2}^{y_3^ht_{h-1}^{-1}}\right]\right)[x,ut_{p-1}^{-1}] =1.
\end{split}
 \end{equation}
Also, from the relations \eqref{eq:rel U1}--\eqref{eq:rel U2} and from \eqref{eq:z1 h}, one computes
\begin{equation}\label{eq:x U}
 \begin{split}
    x^{y_3}&=x^{z_1t_0}=x^{t_0}([z_{d_2},z_{d_2-1}]\cdots[z_3,z_2])^{t_0},\\
  x^{y_3^2} &= x^{t_1}([z_{d_2},z_{d_2-1}]\cdots)^{t_1}
  \left(\left[z_{d_2}^{y_3},z_{d_2-1}^{y_3}\right]\cdots\right)^{t_0^{-1}t_1},\\
  x^{y_3^3} &= x^{t_2}([z_{d_2},z_{d_2-1}]\cdots)^{t_2}
  \left(\left[z_{d_2}^{y_3},z_{d_2-1}^{y_3}\right]\cdots\right)^{t_0^{-1}t_2}
  \left(\left[z_{d_2}^{y_3^2},z_{d_2-1}^{y_3^2}\right]\cdots\right)^{t_1^{-1}t_2},
 \end{split}  \end{equation}
and so on. 
In fact, the two relations \eqref{eq:rel U1}--\eqref{eq:rel U2} --- with the $x^{y_3^h}$'s replaced using \eqref{eq:x U} --- are all the defining relations we need to get $U$, as shown in the following.

\begin{lem}
 The pro-$p$ group $U$ has $\mathrm{r}(U)=2$ defining relations.
\end{lem}

\begin{proof}
 Since $\rmH^n(G,\Z/p\Z)=0$ for every $n\geq3$ (cf. \S~\ref{ssec:quadratic}) and $[G:U]=p$, one has $\rmH^n(U,\Z/p\Z)=0$ for every $n\geq3$ as well (cf. \cite[Prop.~3.3.5]{nsw:cohn}).
 Moreover, one has
 \begin{equation}\label{eq:eulerchar}
 \mathrm{r}(U)-\mathrm{d}(U)+1=p\left(\mathrm{r}(G)-\mathrm{d}(G)+1\right)                      
 \end{equation}
(cf. \cite[Prop.~3.3.13]{nsw:cohn}).
By definition, $\mathrm{r}(G)=2$ and $\mathrm{d}(G)=1+d_1+d_2$, while $\mathrm{d}(U)=3+p(d_1+d_2-2)$ by Lemma~\ref{lem:U ab}.
Therefore, from \eqref{eq:eulerchar} one computes $\mathrm{r}(U)=2$.
\end{proof}


\subsection{The subgroup $H$}
Let $\phi_U\colon U\to\Z/p $ be the homomorphism of pro-$p$ groups defined by $\phi_U(y_1)$, $\phi_U(y_1^{y_3})=-1$, and $\phi_U(w)=0$ for any other element $w$ of $\calY_U$, and put $H=\Ker(\phi_U)$.
Then $H$ is an open subgroup of $U$ of index $p$.
 Set $v=y_1$.
Since $U/H=\{H,vH,\ldots,v^{p-1}H\}$, $H$ is the pro-$p$ group (non-minimally) generated by
\[
 \calX_{H}=\left\{\:v^p,\:\left(vy_1^{y_3}\right)^{v^h},\:w^{v^h}\:\mid w\in \calY_U,\:w\neq v,y_1^{y_3},\:
 h=0,\ldots,p-1\:\:\right\},
\]
and subject to the $2p$ relations $s_1^{v^h}=1$ and $s_2^{v^h}=1$, with $h=0,\ldots,p-1$.
We claim that the abelianization $H^{\ab}$ yields non-trivial torsion.

\begin{prop}
 The abelian pro-$p$ group $H^{\ab}$ is not torsion-free.
\end{prop}

\begin{proof}
 Since all the elements of $\calY_U$ showing up in the last terms of the equalities \eqref{eq:x U} belong to $H$, one deduces that $x^{y_3^h}\equiv x\bmod H'$ for all $h=0,\ldots,p-1$.
 
 Now, each factor of $s_2$ --- cf. \eqref{eq:rel U2} --- is a commutator of elements of $H$, and thus the relations $s_2^{v^h}=1$ yield trivial relations in $H^{\ab}$.
 On the other hand, every factor of $s_1$ --- cf. \eqref{eq:rel U1} ---, but $[x,y_1]$ and $[x^{y_3},y_1^{y_3}]$, is a commutator of elements of $H$.
 From \eqref{eq:rel U1} one obtains
 \begin{equation}\label{eq:rel xv xv-1}
 \left[x^{y_3},y_1^{y_3}\right] [x,y_1]\equiv\left[x,v^{-1}(vy_1^{y_3})\right][x,v]\equiv [x,v^{-1}][x,v]\equiv 1\mod H',
 \end{equation}
as $vy_1^{y_3}\in H$.
Altogether, $H^{\ab}$ is the abelian pro-$p$ group (non-minimally) generated by the set $\calX_{H^{\ab}}=\{wH'\:\mid\:w\in\calX_H\}$, and subject to the $p$ relations
\[
 \left[x^{v^h}H',v^{-1}H'\right]\left[x^{v^h}H',vH'\right]=H',\qquad \text{with }h=0,\ldots,p-1,
\]
as $U/H=\{H,vH,\ldots,v^{p-1}H\}$.
From these relations one deduces the equivalences:
\[ \begin{split}
 x^{v^2} &\equiv \left(x^v\right)^2\cdot x^{-1}\mod H'\qquad\text{with }h=1,\\
 x^{v^3} &\equiv \left(x^{v^2}\right)^2\cdot \left(x^v\right)^{-1}\equiv\left(x^v\right)^3\cdot x^{-2} \mod H'\qquad\text{with }h=2,\\
  &\;\vdots\\
 x^{v^{p-1}} &\equiv \left(x^{v^{p-2}}\right)^2\cdot \left(x^{v^{p-3}}\right)^{-1}\equiv\left(x^v\right)^{p-1}\cdot x^{2-p} \mod H'\qquad\text{with }h=p-2,\\
 x^{v^{p}} &\equiv \left(x^{v^{p-1}}\right)^2\cdot \left(x^{v^{p-2}}\right)^{-1}\equiv\left(x^v\right)^{p}\cdot x^{1-p} \mod H'\qquad\text{with }h=p-1.\\
 \end{split} \]
But $x^{v^p}\equiv x\bmod H'$, as $v^p\in H$, and thus from the last of the above equivalences one obtains
\begin{equation}\label{eq:equiv xv 1}
 x\equiv(x^v)^px^{1-p}\mod H'\;\Longrightarrow\;(x^v)^px^{-p}\equiv (x^vx^{-1})^p\equiv1\mod H'.
\end{equation}
Altogether, $H^{\ab}$ is the abelian pro-$p$ group minimally generated by
\[
 \calY_{H^{\ab}}=\left\{\:v^pH',\: xH', \: x^vH',\:\left(vy_1^{y_3}\right)^{v^h}H',\:w^{v^h}H'\:\mid 
 \: h=0,\ldots,p-1\:\:\right\},
\]
where $w\in\calY_U\smallsetminus\{v,y_1^{y_3},x\}$, and subject to the relation $((xH')^{-1}\cdot x^vH')^p=H'$ --- in particular, $H^{\ab}$ is isomorphic to $\Z_p^{2+p+p^2(d_1+d_2-2)}\times \Z/p\Z$.
\end{proof}


\section{Proof of Theorem~\ref{thm:intro1} case (1.1.b)}\label{sec:thm1}
Let $p$ be an odd prime, and let $G$ be a pro-$p$ group as defined in Theorem~\ref{thm:intro1}, with defining relations as in (1.1.b) --- namely,
\[
  G=\langle\:x,y_1,\ldots,y_{d_1},z_1,\ldots,z_{d_2}\:\mid\:r_1=r_2=1\:\rangle,\]
with
\[\begin{split}
 r_1&= y_1^p[y_1,x]\cdots[y_{d_1-1},y_{d_1}],\\r_2&=z_1^p[z_1,x]\cdots[z_{d_2-1},z_{d_2}].
  \end{split}\]
  

\subsection{Kummerianity}\label{ssec:kummer1}
Let $G_1,G_2$ be the two Demushkin groups as in \S~\ref{ssec:amalg}, with $\epsilon=1$.
By Example~\ref{exam:demushkin}, if 
$$\theta_1\colon G_1\longrightarrow1+p\Z_p\qquad \text{and} \qquad \theta_2\colon G_2\longrightarrow1+p\Z_p$$
are two torsion-free orientations completing respectively $G_1$ and $G_2$ into Kummerian oriented pro-$p$ groups, then necessarily $\theta_1(y_1)=\ldots=\theta_1(y_{d_1})=1$, and analogously $\theta_2(z_1)=\ldots=\theta_1(z_{d_2})=1$,
while $\theta_1(x)=\theta_2(x)=(1-p)^{-1}$.

\begin{prop}\label{lemma:theta}
 An orientation $\theta\colon G\to1+p\Z_p$ completes $G$ into a Kummerian oriented pro-$p$ group $(G,\theta)$ if, and only if, 
 $$\theta(x)=(1-p)^{-1}\qquad\text{and}\qquad \theta(y_i)=\theta(z_j)=1$$ for all $i=1,\ldots,d_1$ and $j=1,\ldots,d_2$.
\end{prop}

\begin{proof}
Suppose that $\theta\colon G\to1+p\Z_p$ is the orientation defined as above, and pick arbitrary $p$-adic integers $\lambda,\lambda_i,\lambda'_j\in\Z_p$ for $1\leq i\leq d_1$ and $1\leq j\leq d_2$.
The assignment $x\mapsto \lambda$, $y_i\mapsto\lambda_i$ and $z_j\mapsto \lambda'_j$ for every $i,j$ yields a well-defined continuous 1-cocycle $c\colon G\to\Z_p(\theta)$, as \eqref{eq:1cocycle comm} imples that
\[\begin{split}
   c(r_1)&= c(y_1^p)+c([y_1,x])+c([y_2,y_3])+\ldots+c([y_{d_1-1},y_{d_1}])\\
   &= p\cdot \lambda_1+ \theta(x)^{-1}(\lambda_1(1-\theta(x))-0)+0+\ldots+0 \\&=0
  \end{split}\]
and 
\[\begin{split}
   c(r_2)&= c(z_1^p)+c([z_1,x])+c([z_2,z_3])+\ldots+c([z_{d_2-1},z_{d_2}])\\
   &= p\cdot \lambda'_1+ \theta(x)^{-1}(\lambda'_1(1-\theta(x))-0)+0+\ldots+0 \\&=0
  \end{split}\]
Therefore, $(G,\theta)$ is  Kummerian by Lemma~\ref{lem:labute inf}.

Conversely, suppose that $(G,\theta)$ is Kummerian.
Let $N_1$ and $N_2$ denote the normal subgroups of $G$ generated as normal subgroups by $z_1,\ldots,z_{d_2}$ and $y_1,\ldots,y_{d_1}$ respectively.
Then $G/N_1\simeq G_1$ and $G/N_2\simeq G_2$.
Moreover, Proposition~\ref{prop:quot} implies that $(G/N_i,\theta_{/N_i})$ is Kummerian for both $i=1,2$.

Since $G/N_i\simeq G_i$ for both $i$, Example~\ref{exam:demushkin} and the argument before the statement of the proposition imply that $\theta_{/N_1}(y_1N_1)=\ldots=\theta_{/N_1}(y_{d_1}N_1)=1$, and analogously 
$\theta_{/N_2}(z_1N_2)=\ldots=\theta_{/N_2}(z_{d_2}N_2)=1$, while $\theta_{/N_1}(xN_1)=\theta_{/N_2}(xN_2)=(1-p)^{-1}$.
Hence, $\theta$ is as defined above, as $\theta(w)=\theta_{/N_1}(wN_1)$ for every $w\in G_1$, and analogously 
$\theta(w)=\theta_{/N_2}(wN_2)$ for every $w\in G_2$.
\end{proof}

Henceforth, $\theta\colon G\to1+p\Z_p$ will denote the orientation as in Proposition~\ref{lemma:theta}.

\subsection{The subgroup $H$}

Let $\phi_1\colon G_1\to\Z/p\oplus\Z/p$ and $\phi_2\colon G_2\to\Z/p\oplus\Z/p$ be the homomorphisms of pro-$p$ groups defined by 
\begin{equation}\label{eq:def hom phi}
 \begin{split}
  & \phi_1(x)=\phi_2(x)=(1,0),\\& \phi_1(y_1)=\phi_2(z_1)=(0,1),\\& \phi_1(y_i)=\phi_2(z_j)=(0,0)\text{ for }i,j\geq2.
  \end{split}
\end{equation}
Put $U_1=\Ker(\phi_1)$ and $U_2=\Ker(\phi_2)$, and also
$$t=z_1^{-1}y_1,\qquad u=x^p,\qquad v=y_1^p,\qquad  w=z_1^p.$$
Then $U_1$ is an open normal subgroup of $G_1$ of index $p^2$, and likewise for $U_2$ and $G_2$ --- note that by \cite{dummitlabute} both $U_1$ and $U_2$ are Demushkin groups.

Finally, put $N_1=\Ker(\theta\vert_{U_1})$, $N_2=\Ker(\theta\vert_{U_2})$, and let $T$ be the subgroup of $G$ generated by $t$.
Observe that $N_1$ and $N_2$ are free pro-$p$ groups, as they are subgroups of infinite index of Demushkin groups (cf. \cite[Ch.~I, \S~4.5, Ex.~5--(b)]{serre:galc}), while $T\simeq \Z_p$ as $G$ is torsion-free (cf. Remark~\ref{rem:torfree}).

Let $H$ be the subgroup of $G$ generated by $U_1$, $U_2$ and $T$, and let $M$ be the subgroup of $H$ generated by $N_1$, $N_2$ and $T$.
Observe that $M\subseteq\Ker(\theta)$.
Our aim is to show that the oriented pro-$p$ group $(H,\theta\vert_H)$ is not Kummerian.
For this purpose, we need the following.

\begin{lem}\label{lem:H noGal}
 \begin{itemize}
  \item[(i)] $M=N_1\amalg N_2\amalg T$.
  \item[(ii)] $M$ is a normal subgroup of $H$, and $H\simeq M\rtimes X^p $
  \item[(iii)] One has an isomorphism of $p$-elementary abelian groups
  \begin{equation}
\frac{G}{\Phi(G)}\simeq \frac{X^p}{X^{p^2}}\times\frac{N_1}{N_1^p[N_1,U_1]}\times \frac{N_2}{N_2^p[N_2,U_2]}\times \frac{T}{T^p}.
  \end{equation}
 \end{itemize}
\end{lem}

\begin{proof}
Consider the pro-$p$ tree $\mathcal{T}$ associated to the amalgamated free pro-$p$ product \eqref{eq:amalg}.  
Namely, $\mathcal{T}$ consists of a set vertices $\mathcal{V}$ and a set of edges $\mathcal{E}$, where
\[\begin{split} 
    \mathcal{V}&=\{\:hG_1,hG_2\:\mid\:h\in G\:\}=G/G_1\:\dot\cup\: G/G_2,\\
    \mathcal{E}&=\{\:hX\:\mid\:h\in G\:\}=G/X,
  \end{split}\]
and it comes endowed with a natural $G$-action, i.e.,
\begin{equation}\label{eq:tree action} 
\begin{split}
  g.(hG_1)=(gh)G_1\qquad &\text{for every } g\in G,\: hG_1\in G/G_1\subseteq \mathcal{V}\\
  g.(hG_1)=(gh)G_2 \qquad&\text{for every }g\in G,\:hG_2\in G/G_2\subseteq \mathcal{V},\\
  g.(hX)=(gh)X\qquad&\text{for every }g\in G,\:hX\in G/X=\mathcal{E}.
\end{split}\end{equation}

Pick $g\in M$ and $hX\in\mathcal{E}$. 
Then $g.hX=hX$ if, and only if, $g\in hXh^{-1}$, i.e., $g=hx^\lambda h^{-1}$ for some $\lambda\in\Z_p$.
Since $M\subseteq\Ker(\theta)$, it follows that
\begin{equation}
1=\theta(g)=\theta\left(hx^\lambda h^{-1}\right)=\theta(x)^\lambda=(1-p)^{\lambda}, 
\end{equation}
and therefore $\lambda=0$, as $1+p\Z_p$ is torsion-free.
Hence, the subgroup $M$ intersects trivially the stabilizer $\mathrm{Stab}_{G}(hX)$ of every edge $hX\in\mathcal{E}$.
By \cite[Thm.~5.6]{melnikov}, $M$ decomposes as free pro-$p$ product as follows:
\begin{equation}\label{eq:melnikov}
M=\left(\coprod_{\omega\in \mathcal{V}'}\mathrm{Stab}_{M}(\omega)\right)\amalg F,
\end{equation}
where $F$ is a free pro-$p$ group, and $\mathcal{V}'\subseteq\mathcal{V}$ is a continuous set of representatives of the space of orbits $M\backslash\mathcal{V}$.
Clearly, the vertices $G_1$ and $G_2$ belong to different orbits, thus in the decomposition \eqref{eq:melnikov} one finds the two factors
\[ \begin{split}
  \mathrm{Stab}_M(G_1)=\{\:g\in M\:\mid\:gG_1=G_1\:\}=M\cap G_1, \\
  \mathrm{Stab}_M(G_2)=\{\:g\in M\:\mid\:gG_2=G_2\:\}=M\cap G_2.
 \end{split}\]
Since $N_1\subseteq M\cap G_1\subseteq \Ker(\theta)\cap G_1=N_1$, one has $\mathrm{Stab}_M(G_1)=N_1$, and analogously $\mathrm{Stab}_M(G_2)=N_2$.
Therefore, from \eqref{eq:melnikov} one obtains 
\begin{equation}\label{eq:melnikov2}
 M=N_1\amalg N_2\amalg \left(\coprod_{\omega\in\mathcal{V}'\smallsetminus\{G_1,G_2\}}\mathrm{Stab}_{M}(\omega)\amalg F\right).
\end{equation}
It is straightforward to see that $t\notin N_1\amalg N_2$. 
Since $M$ is generated as pro-$p$ group by $N_1$, $N_2$ and $t$, the right-side factor in \eqref{eq:melnikov2} is necessarily $T$, and this proves (i).

In order to prove (ii), we need only to show that $uMu^{-1}=M$, as $H=\langle\:u,M\:\rangle$.
Since $N_1$ is normal in $U_1$, and $u\in U_1$, then $uN_1u^{-1}=N_1$ --- analogously, $uN_2u^{-1}=N_2$.
Now, observe that the integer
\[
 (1-p)^p-1=\left(1-\binom{p}{1}p+\binom{p}{2}p^2-\ldots-p^p\right)-1
\]
is divisible by $p^2$ but not by $p^3$, so we put $(1-p)^p=1+p^2\lambda$, with $\lambda\in1+p\Z_p$.
From the relation $r_1=1$ one deduces
\begin{equation}\label{eq:action x y}
 y_1^{x}=y_1^{1-p}\cdot\left([y_2,y_3]\cdots[y_{d_1-1},y_{d_1}]\right)^{-1},
\end{equation}
and by iterating \eqref{eq:action x y} $p$ times, one obtains $ y_1^{u}=y_1^{(1-p)^p} n_1$
for some $n_1\in N_1'$ --- for this purpose, observe that for every $\nu\geq0$ and $i\geq1$, the triple commutator 
$$\left[y_1^\nu,[y_i,y_{i+1}]\right]=\left[y_i^{y_1^\nu},y_{i+1}^{y_1^\nu}\right]^{-1}\cdot [y_i,y_{i+1}]$$
belongs to $N_1'$, as $y_i^{y_0^\nu}\in N_1$.
Analogously, $ z_1^{u}=z_1^{(1-p)^p}n_2$ for some $n_2\in N_2'$.
Altogether, 
\begin{equation}\label{eq:action u t}
 t^{u}=(z_1^{-1}y_1)^{u}=z_1^{u}y_1^{u}=
 n_2^{-1}\cdot w^{-p\lambda}\cdot t\cdot v^{p\lambda}\cdot n_1,
\end{equation}
which belongs to $M$ --- 
here we replaced $z_1^{-(1-p)^p}=w^{-p\lambda}\cdot z_1^{-1}$ and $y_1^{(1-p)^p}=y_1\cdot v^{p^\lambda}$. 
Hence, $M\unlhd H$.
Finally, by definition $H=M\cdot X^p$, and moreover $$M\cap X^p\subseteq\Ker(\theta)\cap X^p=\{1\},$$ so that $H=M\rtimes X^p$.
This completes the proof of (ii).

Finally, by (i) and (ii) one has the isomorphism of $p$-elementary abelian groups
\begin{equation}\label{eq:isom H Phi}
\begin{split}
 M/\Phi(M) &\simeq N_1/\Phi(N_1)\times N_2/\Phi(N_2)\times T/T^p \\
 H/\Phi(H) &\simeq X^p/X^{p^2}\times M/M^p[M,H].
\end{split}
\end{equation}
From \eqref{eq:action u t} one has that $[T,X^p]\subseteq \Phi(M)$, and since $H=MX^p$, $U_1=N_1X^p$, and $U_2=N_2X^p$, form \eqref{eq:isom H Phi} one deduces (iii).
\end{proof}


\subsection{The subgroup $H$ and Kummerianity}

\begin{prop}
 The oriented pro-$p$ group $(H,\theta\vert_H)$ is not Kummerian.
\end{prop}

\begin{proof}
Let $N$ be the normal subgroup of $H$ generated as a normal subgroup by $N_1,N_2$, and set $\bar H=H/N$.
Then $N\subseteq\Ker(\theta\vert_H)$, and clearly $\bar H$ is finitely generated.
Moreover, by duality the restriction map 
$\mathrm{res}_{H,N}^1\colon H^1(H,\Z/p\Z)\to H^1(N,\Z/p\Z)^H$ is surjective, as by Lemma~\ref{lem:H noGal} one has 
$$N/N^p[N,H]\simeq N_1/N_1^p[N_1,U_1]\times N_2/N_2^p[N_2,U_2],$$
which embeds in $H/\Phi(H)$.
In particular, $\{uN,tN\}$ is a minimal generating set of $\bar H$.
Thus, by Proposition~\ref{prop:quot} if the oriented pro-$p$ group $(\bar H,\bar\theta)$ is not  Kummerian --- where $\bar\theta=(\theta\vert_H)_{/N}\colon \bar H\to1+p\Z_p$ is the orientation induced by $\theta\vert_H$ ---, then also $(H,\theta\vert_H)$ is not  Kummerian.

By \eqref{eq:action u t}, in $H$ one has that $[t,u^{-1}]\equiv 1\bmod N$, and thus $\bar H$ is abelian.
Moreover, 
$$\bar\theta(uN)=\theta(u)=(1-p)^p \qquad\text{and}\qquad \bar\theta(tN)=\theta(t)=1,$$
so that $\Ker(\bar\theta)=\langle tN\rangle$.
Therefore, the subgroup $K_{\bar\theta}(\bar H)$ is generated by
$$\left(t^{-\theta(u)}utu^{-1}\right)N=t^{p^2\lambda}N.$$
Thus, the quotient $\Ker(\bar\theta)/K_{\bar\theta}(\bar H)=\langle tN\rangle/\langle tN\rangle^{p^2}$ is not torsion-free, and by Proposition~\ref{prop:kummer}, $(\bar H,\bar\theta)$ is not  Kummerian.
\end{proof}

This completes the proof of Theorem~\ref{thm:intro1} case~(1.1.b).

\begin{rem}\rm
 If $d_1=d_2=1$, case~(1.1.b) of Theorem~\ref{thm:intro1} is a particular case of \cite[Prop.~6.5]{BQW}.
\end{rem}



\section{Massey products}

\subsection{Massey products in Galois cohomology}
\label{ssec:massey}

Here we recall briefly what we need in order to prove Proposition~\ref{prop:properties intro}.
For a detailed account on Massey products for pro-$p$ groups, we direct the reader to \cite{vogel,ido:massey,mt:massey}.

Let $G$ be a pro-$p$ group.
For $n\geq2$, the {\sl $n$-fold Massey product} on $\rmH^1(G,\Z/p\Z)$ is a multi-valued map
\[
 \underbrace{\rmH^1(G,\Z/p\Z)\times \ldots\times \rmH^1(G,\Z/p\Z)}_{n\text{ times}}\longrightarrow \rmH^2(G,\Z/p\Z).
\]
For $n\geq2$, given a sequence $\alpha_1,\ldots,\alpha_n$ of elements of $\rmH^1(G,\Z/p\Z)$ (with possibly $\alpha_i=\alpha_j$ for some $1\leq i<j\leq n$), the (possibly empty) subset of $\rmH^2(G,\Z/p\Z)$ which is the value of the $n$-fold Massey product associated to the sequence $\alpha_1,\ldots,\alpha_n$ is denoted by $\langle\alpha_1,\ldots,\alpha_n\rangle$.
If $n=2$, then the 2-fold Massey product coincides with the cup-product, i.e., for $\alpha_1,\alpha_2\in \rmH^1(G,\Z/p\Z)$ one has 
\begin{equation}\label{eq:cup 2massey}
 \langle\alpha_1,\alpha_2\rangle=\{\alpha\smallsmile\alpha_2\}\subseteq \rmH^2(G,\Z/p\Z).
\end{equation}

A pro-$p$ group $G$ is said to satisfy:
\begin{itemize}
 \item[(a)] the {\sl $n$-Massey vanishing property} (with respect to $\Z/p\Z$) if for every 
sequence $\alpha_1,\ldots,\alpha_n$ of elements of $\rmH^1(G,\Z/p\Z)$, $\langle\alpha_1,\ldots,\alpha_n\rangle\neq\varnothing$ implies $0\in \langle\alpha_1,\ldots,\alpha_n\rangle$;
\item[(b)] the {\sl strong $n$-Massey vanishing property} (with respect to $\Z/p\Z$) if for every sequence $\alpha_1,\ldots,\alpha_n$ of elements of $\rmH^1(G,\Z/p\Z)$, the condition on the cup-products
\begin{equation}\label{eq:cup trivial}
 \alpha_1\smallsmile\alpha_2=\alpha_2\smallsmile\alpha_3=\ldots=\alpha_{n-1}\smallsmile\alpha_n=0
\end{equation}
implies $0\in \langle\alpha_1,\ldots,\alpha_n\rangle$ (cf. \cite[Def.~1.2]{pal:massey})
--- we remind that the triviality condition \eqref{eq:cup trivial} is satisfied whenever $\langle\alpha_1,\ldots,\alpha_n\rangle\neq\varnothing$, cf., e.g., \cite[\S~2]{mt:massey};
\item[(c)] the {\sl cyclic $p$-Massey vanishing property} if for every element $\alpha\in\rmH^1(G,\Z/p\Z)$, the $p$-fold Massey product $\langle\alpha,\ldots,\alpha\rangle$ contains 0.
\end{itemize}

\begin{rem}\label{rem:massey cup}\rm
Given a sequence $\alpha_1,\ldots,\alpha_n$ of elements of $\rmH^1(G,\Z/p\Z)$, if an element $\omega$ of $\rmH^2(G,\Z/p\Z)$ is a value of the $n$-fiold Massey product $\langle\alpha_1,\ldots,\alpha_n\rangle$, then 
$$\omega+\alpha_1\smallsmile\beta\in \langle\alpha_1,\ldots,\alpha_n\rangle
\qquad\text{and}\qquad\omega+\alpha_n\smallsmile\beta\in \langle\alpha_1,\ldots,\alpha_n\rangle$$
 for any $\beta\in\rmH^1(G,\Z/p\Z)$ (cf. \cite[Rem.~2.2]{mt:massey}).
\end{rem}

In \cite[Thm.~8.1]{MT:ker}, J.~Mina\v c and N.D.~T\^an proved that the maximal pro-$p$ Galois group of a field $\K$ containing a root of 1 of order $p$ (and also $\sqrt{-1}$ if $p=2$) satisfies the cyclic $p$-Massey vanishing property.
The proof of the last property for a pro-$p$ group $G$ as in Theorem~\ref{thm:intro1} is rather immediate.

\begin{proof}[Proof of Proposition~\ref{prop:properties intro}--(ii)]
 By Proposition~\ref{prop:thm2 kummer} and Proposition~\ref{lemma:theta}, $G$ may complete into a Kummerian oriented pro-$p$ group with torsion-free orientation.
 Hence, $G$ satisfies the cyclic $p$-Massey vanishing property by \cite[Thm.~3.10]{cq:massey}.
\end{proof}


\subsection{Massey products and unipotent upper-triangular matrices}\label{ssec:massey unip}

Massey products for a pro-$p$ group $G$ may be translated in terms of unipotent upper-triangular representations of $G$ as follows.
For $n\geq 2$ let
\[
 \dbU_{n+1}=\left\{\left(\begin{array}{ccccc} 1 & a_{1,2} & \cdots & & a_{1,n+1} \\ & 1 & a_{2,3} &  \cdots & \\
 &&\ddots &\ddots& \vdots \\ &&&1& a_{n,n+1} \\ &&&&1 \end{array}\right)\mid a_{i,j}\in\Z/p \right\}\subseteq 
 \mathrm{GL}_{n+1}(\Z/p\Z)
\]
be the group of unipotent upper-triangular $(n+1)\times(n+1)$-matrices over $\Z/p$.
Then $\dbU_{n+1}$ is a finite $p$-group.
Moreover, for $1\leq h,l\leq n+1$ let $E_{h,l}$ denote the $(n+1)\times(n+1)$ matrix with the $(h,l)$-entry equal to 1, and all the other entries equal to $0$.

Now let $\rho\colon G\to\dbU_{n+1}$ be a homomorphism of pro-$p$ groups.
Observe that for every $h=1,\ldots,n$, the projection $\rho_{h,h+1}\colon G\to\Z/p$ of $\rho$ onto the $(h,h+1)$-entry is a homomorphism, and thus it may be considered as an element of $\rmH^1(G,\Z/p\Z)$.
One has the following ``pro-$p$ translation'' of a result of W.~Dwyer which interprets Massey product in terms of unipotent upper-triangular representations (cf., e.g., \cite[Lemma~9.3]{eq:kummer}).

\begin{prop}\label{prop:representations}
 Let $G$ be a pro-$p$ group, and let $\alpha_1,\ldots,\alpha_n$ be a sequence of elements of $\rmH^1(G,\Z/p\Z)$, with $n\geq2$.
 Then the $n$-fold Massey product $\langle\alpha_1,\ldots,\alpha_n\rangle$:
 \begin{itemize}
  \item[(i)] is not empty if, and only if, there exists a morphism of pro-$p$ groups $\bar\rho\colon G\to\dbU_{n+1}/\mathrm{Z}(\dbU_{n+1})$ such that $\bar\rho_{h,h+1}=\alpha_h$ for every $h=1,\ldots,n$;
  \item[(ii)] vanishes if, and only if, there exists a morphism of pro-$p$ groups $\rho\colon G\to\dbU_{n+1}$ such that $\rho_{h,h+1}=\alpha_h$ for every $h=1,\ldots,n$.
 \end{itemize}
\end{prop}

We recall that
\[
 \mathrm{Z}(\dbU_{n+1})=\left\{\:I_{n+1}+aE_{1,n+1}\:\mid\:a\in\Z/p\Z\:\right\}\simeq\Z/p\Z.
\]
We use this fact to prove statements (iii.a)--(iii.b) of Proposition~\ref{prop:properties intro}.
First of all, let $G$ be as in Theorem~\ref{thm:intro1}, and let $\alpha_1,\ldots,\alpha_n$ be a sequence of elements of $\rmH^1(G,\Z/p\Z)$.
Keeping the same notation as in \S~\ref{ssec:quadratic}, for $h=1,\ldots,n$ one has 
\[
\alpha_h=\alpha_h(x)\cdot\chi+\sum_{i=1}^{d_1}\alpha_h(y_i)\cdot\varphi_i+\sum_{j=1}^{d_2}\alpha_h(z_j)\cdot\psi_j. \]
Therefore, for $h=1,\ldots,n-1$ one obtains 
$$\alpha_h\smallsmile\alpha_h=S_h\cdot(\chi\smallsmile\varphi_1)+S'_h\cdot(\chi\smallsmile\psi_1),$$ 
where
\[
 \begin{split}
  S_h= & (\alpha_h(x)\alpha_{h+1}(y_1)-\alpha_{h}(y_1)\alpha_{h+1}(x))+\\
  &\;+(-1)^\epsilon  \sum_{2\mid i} (\alpha_h(y_i)\alpha_{h+1}(y_{i+1})-\alpha_{h}(y_{i+1})\alpha_{h+1}(y_i)), \\
  S'_h =&  (\alpha_h(x)\alpha_{h+1}(z_1)-\alpha_{h}(z_1)\alpha_{h+1}(x))+\\
  &\;+(-1)^{\epsilon}\sum_{ 2\mid j} (\alpha_h(z_j)\alpha_{h+1}(z_{j+1})-\alpha_{h}(z_{j+1})\alpha_{h+1}(z_j)),
 \end{split}\]
 with $\epsilon=0$ if $G$ is as in (1.1.a), and  $\epsilon=1$ if $G$ is as in (1.1.b).
If the sequence $\alpha_1,\ldots,\alpha_n$ satisfies condition~\eqref{eq:cup trivial}, then one has $S_h=S'_h=0$ for $h=1,\ldots,n-1$, as $\{\chi\smallsmile\varphi_1,\chi\smallsmile\psi_1\}$ is a basis of $\rmH^2(G,\Z/p)$.

From now on, we will assume that $p>3$ while considering a pro-$p$ group $G$ as in (1.1.b), unless stated otherwise.



\subsection{3-fold Massey products}

We are ready to prove the following.

 \begin{prop}\label{prop:strong massey 3}
A pro-$p$ group $G$ satisfies the 3-Massey vanishing property in the following cases:
\begin{itemize}
 \item[(a)] if $G$ is as in {\rm (1.1.a)};
 \item[(b)] if $G$ is as in {\rm (1.1.b)} and $p>3$.
\end{itemize}  
 \end{prop}

\begin{proof}
 Let $\alpha_1,\alpha_2,\alpha_3$ be a sequence of elements of $\rmH^1(G,\Z/p\Z)$ satisfying \eqref{eq:cup trivial}.
 Then $S_1=S_1'=S_2=S_2'=0$ (cf. \S~\ref{ssec:massey unip}).
 Our goal is to construct a morphism $\rho\colon G\to\dbU_4$ such that $\rho_{1,2}=\alpha_1$, $\rho_{2,3}=\alpha_2$, $\rho_{3,4}=\alpha_3$.
 
 For every $w\in\calX$ set 
 \[
  A(w)=I+\alpha_1(w)E_{1,2}+\alpha_2(w)E_{2,3}+\alpha_3(w)E_{3,4}\in\dbU_4,
 \]
where $I$ denotes the $4\times4$ identity matrix.
If $G$ is as in (1.1.a), then one computes
\begin{equation}\label{eq:C 3 a}
  \begin{split}
  C&=[A(x),A(y_1)]\cdots[A(y_{d_1-1}),A(y_{d_1})] \\
   &= I+E_{1,4}\left(\alpha_1(y_1)\alpha_2(x)\alpha_3(y_1)+\sum_{2\mid i}\alpha_1(y_i)\alpha_2(y_{i+1})\alpha_3(y_i)\right)\\
C'&=[A(x),A(z_1)]\cdots[A(z_{d_2-1}),A(z_{d_2})] \\
   &= I+E_{1,4}\left(\alpha_1(z_1)\alpha_2(x)\alpha_3(z_1)+\sum_{2\mid j}\alpha_1(z_j)\alpha_2(z_{j+1})\alpha_3(z_j)\right);
 \end{split}\end{equation}
while if $G$ is as in (1.1.b), then one computes
\begin{equation}\label{eq:C 3 b}
  \begin{split}
  C&=A(y_1)^p[A(y_1),A(x)]\cdots[A(y_{d_1-1}),A(y_{d_1})] \\
   &= I+E_{1,4}\left(\alpha_1(x)\alpha_2(y_1)\alpha_3(x)+\sum_{2\mid i}\alpha_1(y_i)\alpha_2(y_{i+1})\alpha_3(y_i)\right)\\
C'&=A(z_1)^p[A(z_1),A(x)]\cdots[A(z_{d_2-1}),A(z_{d_2})] \\
   &= I+E_{1,4}\left(\alpha_1(x)\alpha_2(z_1)\alpha_3(x)+\sum_{2\mid j}\alpha_1(z_j)\alpha_2(z_{j+1})\alpha_3(z_j)\right).
 \end{split}\end{equation}
 --- observe that the exponent of $\dbU_4$ is $p$, as $p>4$, and thus $A(y_1)^p=A(z_1)^p=I$.

In both cases, $C,C'\in\mathrm{Z}(\dbU_4)$, and therefore the assignment $w\mapsto A(w)$ for every $w\in\calX$ yields a morphism $\bar\rho\colon G\to\dbU_4/\mathrm{Z}(\dbU_4)$ satisfying $\bar\rho_{h,h+1}=\alpha_h$ for $h=1,2,3$.
Thus, $\langle\alpha_1,\alpha_2,\alpha_3\rangle\neq\varnothing$ by Proposition~\ref{prop:representations}.

Moreover, if $C=C'=I$ then the same assignment yields a morphism $\rho\colon G\to\dbU_4$ with the desired properties.
In particular, by \eqref{eq:C 3 a}--\eqref{eq:C 3 b} one has $C=I$ if $\alpha_1(w)=\alpha_3(w)=0$ for every $w=y_1,\ldots,y_{d_1}$, or for every $w=y_2,\ldots,y_{d_1}$ and $w=x$; and analogously $C'=I$ if $\alpha_1(w)=\alpha_3(w)=0$ for every $w=z_1,\ldots,z_{d_{d_2}}$, or for every $w=z_2,\ldots,z_{d_2}$ and $w=x$.

On the other hand, if $C\neq I$ then $\chi\smallsmile\varphi_1=\pm\trg(r_1 G_{(3)})$ belongs to $\langle\alpha_1,\alpha_2,\alpha_3\rangle$, and analogusly if $C'\neq I$ then $\chi\smallsmile\psi_1=\pm\trg(r_2 G_{(3)})$ belongs to $\langle\alpha_1,\alpha_2,\alpha_3\rangle$ (cf. \cite[Lemma~3.7]{mt:massey}) --- here the sign depends on whether the relations are as in (1.1.a) or in (1.1.b).
Now, if $\alpha_h(y_i)\neq 0$ for some $h=1,3$ and $i\in\{2,\ldots,d_1\}$, then 
$$\chi\smallsmile\varphi_1=\alpha_h\smallsmile\beta\qquad\text{for some }\beta\in\rmH^1(G,\Z/p\Z).$$
Analogously, if $\alpha_h(z_j)\neq 0$ for some $h=1,3$ and $j\in\{2,\ldots,d_2\}$, then 
$$\chi\smallsmile\psi_1=\alpha_h\smallsmile\beta\qquad\text{for some }\beta\in\rmH^1(G,\Z/p\Z).$$
Moreover, if $\alpha_h(x)\neq 0$ for some $h=1,3$, then 
$$\chi\smallsmile\varphi_1=\alpha_h\smallsmile\beta\qquad\text{and}\qquad\chi\smallsmile\psi_1=\alpha_h\smallsmile\beta'$$
for some $\beta,\beta'\in\rmH^1(G,\Z/p\Z)$.
Therefore, Remark~\ref{rem:massey cup} implies that if $C\neq I$ or $C'\neq I$ then $0\in\langle\alpha_1,\alpha_2,\alpha_3\rangle$ anyway.
\end{proof}

\begin{rem}\label{rem:3MP}\rm
If $p=3$ and $G$ as in (1.1.b), then $G$ does not satisfy the 3-Massey vanishing property.
Indeed, set $\alpha_1=\alpha_3=\varphi_1+\psi_1$, and $\alpha_2=\varphi_1$.
Then $$\alpha_1\smallsmile\alpha_2=\alpha_2\smallsmile\alpha_3=\pm(\varphi_1\smallsmile\psi_1)=0.$$
It is easy to see that one may construct a morphism of pro-$p$ groups $\bar\rho\colon G\to\dbU_4/\mathrm{Z}(\dbU_4)$ such that $\bar\rho_{1,2}=\bar\rho_{3,4}=\alpha_1$ and $\bar\rho_{2,3}=\alpha_2$ --- and thus $\langle\alpha_1,\alpha_2,\alpha_1\rangle\neq\varnothing$ by Proposition~\ref{prop:representations} ---; but, on the other hand, one may not construct a morphism of pro-$p$ groups $\rho\colon G\to\dbU_4$ satisfying $\rho_{1,2}=\rho_{3,4}=\alpha_1$ and $\rho_{2,3}=\alpha_2$ --- so that $0\notin\langle\alpha_1,\alpha_2,\alpha_1\rangle$ by Proposition~\ref{prop:representations}.
\end{rem}


\subsection{4-fold Massey products}

 \begin{prop}\label{prop:strong massey 4}
A pro-$p$ group $G$ as in Theorem~\ref{thm:intro1} satisfies the strong 4-Massey vanishing property.
 \end{prop}
 
\begin{proof}
Let $\alpha_1,\ldots,\alpha_4$ be a sequence of four elements of $\rmH^1(G,\Z/p\Z)$ satisfying
\eqref{eq:cup trivial}.
Our goal is to construct a homomorphism of pro-$p$ groups $\rho\colon G\to\dbU_5$ such that $\rho_{h,h+1}=\alpha_h$ for $h=1,\ldots,5$, so that the claim follows by Proposition~\ref{prop:representations}.

Let $I$ denote the identity matrix of the group $\dbU_5$.
For every $w\in\calX=\{x,y_1,\ldots,z_{d_2}\}$ set 
\[
A(w)=\left( \begin{array}{ccccc} 1 &\alpha_1(w)&0&0&0 \\ &1&\alpha_2(w)&0&0 \\ &&1&\alpha_3(w)&0 \\ &&&1&\alpha_4(w) \\ &&&&1  \end{array}  \right)\in\dbU_5. 
\]
Moreover, put
\[\begin{split} 
C &=(c_{hl})= A(y_1)^{\epsilon p}\cdot[A(x),A(y_1)]^{(-1)^\epsilon}\cdots\left[A(y_{d_1-1}),A(y_{d_1})\right],\\
C'&=(c'_{hl})= A(z_1)^{\epsilon p}\cdot[A(x),A(z_1)]^{(-1)^\epsilon}\cdots\left[A(z_{d_2-1}),A(z_{d_2})\right].
\end{split}\]
We will consider the matrix $C$ as a function of the matrices $A(x),\ldots,A(y_{d_1})$, and the matrix $C'$ as a function of the matrices $A(x),A(z_1),\ldots,A(z_{d_2})$.

Since $p\geq5$, the exponent of the $p$-group $\dbU_5$ is $p$, and thus $A(y_1)^p=A(z_1)^p=I$.
Moreover, for every $w,w'\in\calX$, the $(h,h+1)$-entry of $[A(w),A(w')]$ is 0 for every $h=1,\ldots,4$, and thus also 
$c_{h,h+1}=c'_{h,h+1}=0$.
Moreover, for $h=1,2,3$ one has $c_{h,h+2}=S_h$ and $c'_{h,h+2}=S'_h$ --- which are equal to 0 by \eqref{eq:cup trivial}.

We split the proof in the analysis of the following three cases.
Our aim is to modify suitably the matrices $A(w)$ --- without modifying the $(h,h+1)$-entries with $h=1,\ldots,4$ --- in order to obtain $C=C'=I$.

\medskip

\noindent {\bf Case 1.} Suppose first that:
\begin{itemize}
 \item[(1.a)] $\alpha_2(x)=\alpha_2(y_i)=0$ for all $2\leq i\leq d_1$; or
\item[(1.b)] $\alpha_3(x)=\alpha_3(y_i)=0$ for all $2\leq i\leq d_1$.
\end{itemize}
Since $S_1=S_2=S_3=0$ by \eqref{eq:cup trivial}, one has
\begin{eqnarray}\label{eq:0 cond 2a}
 \alpha_1(x)\alpha_2(y_1)=\alpha_2(y_1)\alpha_3(x) &=&0, \\
 \alpha_2(x)\alpha_3(y_1)=\alpha_3(y_1)\alpha_4(x) &=&0,
\label{eq:0 cond 2b}
\end{eqnarray}
respectively in case (1.a) and in case (1.b).
Applying \eqref{eq:0 cond 2a}--\eqref{eq:0 cond 2b}, one computes 
\[
[A(y_1),A(x)]=\begin{cases}
           I+\left(\alpha_3(y_1)\alpha_4(x)-\alpha_3(x)\alpha_4(y_1)\right)E_{3,5} &\text{in case (1.a)}, \\
           I+\left(\alpha_1(y_1)\alpha_2(x)-\alpha_2(x)\alpha_1(y_1)\right)E_{1,3} &\text{in case (1.b)}, 
          \end{cases}\]
and
\[
[A(y_i),A(y_{i+1})]=\begin{cases}
           I+\left(\alpha_3(y_i)\alpha_4(y_{i+1})-\alpha_3(y_{i+1})\alpha_4(y_i)\right)E_{3,5} &\text{in case (1.a)}, \\
           I+\left(\alpha_1(y_i)\alpha_2(y_{i+1})-\alpha_2(y_{i+1})\alpha_1(y_i)\right)E_{1,3} &\text{in case (1.b)}, 
          \end{cases}\]
for $i=2,4,\ldots,d_1-1$.          
Altogether, one has $C=I+S_3E_{3,5}$ in case (1.a) and $C=I+S_1E_{1,3}$ in case (1.b), so that in both cases $C=I$ by \eqref{eq:cup trivial}.

\noindent Analogously, if $\alpha_2(x)=\alpha_2(z_j)=0$ for all $2\leq j\leq d_2$, or if $\alpha_3(x)=\alpha_3(z_j)=0$ for all $2\leq j\leq d_2$, then $C'=I$.
This completes the analysis of case 1.

\medskip

\noindent {\bf Case 2.}
Now suppose that $\alpha_1(x)=\alpha_4(x)=\alpha_1(y_i)=\alpha_4(y_i)=0$ for all $2\leq i\leq d_1$.
Since $S_1=S_2=S_3=0$ by \eqref{eq:cup trivial}, one has
\begin{equation}\label{eq:0 cond 3}
 \alpha_1(y_1)\alpha_2(x)=\alpha_3(x)\alpha_4(y_1)=0.
\end{equation}
Then one computes
\[
 \begin{split}
  [A(y_1),A(x)] &= I+\left(\alpha_2(y_1)\alpha_3(x)-\alpha_2(x)\alpha_3(y_1)\right)E_{2,4}+\alpha_2(x)\alpha_3(y_1)\alpha_4(y_1)E_{2,5},\\
  [A(y_i),A(y_{i+1})]&=I+\left(\alpha_2(y_i)\alpha_3(y_{i+1})-\alpha_2(y_{i+1})\alpha_3(y_i)\right)E_{2,4},
 \end{split}
\]
where we apply \eqref{eq:0 cond 3} to obtain the first equality, and in the second one $i$ runs through the even positive integers between $2$ and $d_1-1$.
If $\alpha_2(x)\alpha_3(y_1)\alpha_4(y_1)=0$ then it is straightforward to see that $C=I+S_2E_{2,4}=I$.
Otherwise, $\alpha_2(x)\neq0$, so that \eqref{eq:0 cond 3} implies that $\alpha_1(y_1)=0$.
In this case, set $$\tilde A=I-\alpha_3(y_1)\alpha_4(y_1)E_{3,5}.$$
Then 
\[
 \left[\tilde A,A(x)\right]=I-\alpha_2(x)\alpha_3(y_1)\alpha_4(y_1)E_{2,5}, 
\]
and 
\[\begin{split}
 \left[A(y_1)\tilde A,A(x)\right] &= \underbrace{\left[A(y_1),[\tilde A,A(x)]\right]}_{=I}\left[\tilde A,A(x)\right][A(y_1),A(x)]\\
 &= I+\left(\alpha_2(y_1)\alpha_3(x)-\alpha_2(x)\alpha_3(y_1)\right)E_{2,4}.   
  \end{split}
\]
Therefore, replacing $A(y_1)$ with $A(y_1)\tilde A$ yields $c_{2,4}=S_2=0$ and $C_{hl}=0$ for $h<l$, i.e., $C=I$.

\noindent An analogous argument yields $C'=I$ --- after replacing suitably the matrix $A(z_1)$ if needed --- if $\alpha_1(x)=\alpha_3(x)=\alpha_1(z_j)=\alpha_3(z_j)=0$ for all $1\leq j\leq d_2$.
This completes the analysis of case 2.

\medskip

\noindent {\bf Case 3.}
Finally, if none of the above two assumptions on the triviality of the values $\alpha_h(x)$ and $\alpha_h(y_i)$, with $2\leq i\leq d_1$, hold true, then
\begin{itemize}
 \item[(3.a)] there are $w,w'\in\{x,y_2,\ldots,y_{d_1}\}$ --- possibly $w=w'$ --- such that $\alpha_1(w)\neq0$ and $\alpha_2(w')\neq0$, or
 \item[(3.b)] there are $w,w'\in\{x,y_2,\ldots,y_{d_1}\}$ --- possibly $w=w'$ --- such that $\alpha_4(w)\neq0$ and $\alpha_3(w')\neq0$.
\end{itemize}
Suppose we are in case (3.a).
If $w=x$ or $w=y_i$ with $i$ odd, set 
$$\tilde A=\begin{cases}
            I+\frac{c_{1,4}}{\alpha_1(w)}E_{2,4}& \text{if }w\in\{\:x,y_3,\ldots,y_{d_1}\:\}\\
            I-\frac{c_{1,4}}{\alpha_1(w)}E_{2,4}& \text{if }w\in\{\:y_i\:\mid\:i\text{ is even}\:\},
           \end{cases}$$
and replace $A(y_1)$ with $A(y_1)\tilde A$, if $w=x$, or $A(y_{i-1})$ with $A(y_{i-1})\tilde A$ if $w=y_i$ with $i$ odd, or $A(y_{i+1})$ with $A(y_{i+1})\tilde A$, if $w=y$ with $i$ even.
After the replacement, one has $c_{hl}=0$ for $h< l\leq h+2$, and for $(h,l)=(1,4)$.
Then, set
$$\tilde A'=\begin{cases}
            I+\frac{c_{2,5}}{\alpha_1(w')}E_{3,5}& \text{if }w'\in\{\:x,y_3,\ldots,y_{d_1}\:\}\\
            I-\frac{c_{2,5}}{\alpha_1(w')}E_{3,5}& \text{if }w'\in\{\:y_i\:\mid\:\text{ is even}\:\},
           \end{cases}$$
and replace $A(y_1)$ with $A(y_1)\tilde A'$, if $w=x$, or $A(y_{i-1})$ with $A(y_{i-1})\tilde A'$ if $w=y_i$ with $i$ odd, or $A(y_{i+1})$ with $A(y_{i+1})\tilde A'$, if $w=y$ with $i$ even.
After this further replacement, one has $c_{hl}=0$ for $h< l\leq h+3$.
Finally, set 
$$\tilde A''=\begin{cases}
            I+\frac{c_{1,5}}{\alpha_1(w)}E_{2,5}& \text{if }w\in\{\:x,y_3,\ldots,y_{d_1}\:\}\\
            I-\frac{c_{1,5}}{\alpha_1(w)}E_{2,5}& \text{if }w\in\{\:y_i\:\mid\:i\text{ is even}\:\},
           \end{cases}$$
and replace $A(y_1)$ with $A(y_1)\tilde A''$, if $w=x$, or $A(y_{i-1})$ with $A(y_{i-1})\tilde A''$ if $w=y_i$ with $i$ odd, or $A(y_{i+1})$ with $A(y_{i+1})\tilde A''$, if $w=y$ with $i$ even.
After this last replacement, one has $C=I$.

\medskip

\noindent Now suppose we are in case (3.b).
If $w=x$ or $w=y_i$ with $i$ odd, set 
$$\tilde A=\begin{cases}
            I-\frac{c_{2,5}}{\alpha_4(w)}E_{3,4}& \text{if }w\in\{\:x,y_3,\ldots,y_{d_1}\:\}\\
            I+\frac{c_{2,5}}{\alpha_4(w)}E_{3,4}& \text{if }w\in\{\:y_i\:\mid\:  i\text{ is even}\:\},
           \end{cases}$$
and replace $A(y_1)$ with $A(y_1)\tilde A$, if $w=x$, or $A(y_{i-1})$ with $A(y_{i-1})\tilde A$ if $w=y_i$ with $i$ odd, or $A(y_{i+1})$ with $A(y_{i+1})\tilde A$, if $w=y$ with $i$ even.
After the replacement, one has $c_{hl}=0$ for $h< l\leq h+2$, and for $(h,l)=(2,5)$.
Then, set
$$\tilde A'=\begin{cases}
            I-\frac{c_{1,4}}{\alpha_3(w')}E_{1,3}& \text{if }w'\in\{\:x,y_3,\ldots,y_{d_1}\:\}\\
            I+\frac{c_{1,4}}{\alpha_3(w')}E_{1,3}& \text{if }w'\in\{\:y_i\:\mid\:i\text{ is even}\:\},
           \end{cases}$$
and replace $A(y_1)$ with $A(y_1)\tilde A'$, if $w=x$, or $A(y_{i-1})$ with $A(y_{i-1})\tilde A'$ if $w=y_i$ with $i$ odd, or $A(y_{i+1})$ with $A(y_{i+1})\tilde A'$, if $w=y$ with $i$ even.
After this further replacement, one has $c_{hl}=0$ for $h< l\leq h+3$.
Finally, set 
$$\tilde A''=\begin{cases}
            I-\frac{c_{1,5}}{\alpha_1(w)}E_{1,4}& \text{if }w\in\{\:x,y_3,\ldots,y_{d_1}\:\}\\
            I+\frac{c_{1,5}}{\alpha_1(w)}E_{1,4}& \text{if }w\in\{\:y_i\:\mid\: i\text{ is even}\:\},
           \end{cases}$$
and replace $A(y_1)$ with $A(y_1)\tilde A''$, if $w=x$, or $A(y_{i-1})$ with $A(y_{i-1})\tilde A''$ if $w=y_i$ with $i$ odd, or $A(y_{i+1})$ with $A(y_{i+1})\tilde A''$, if $w=y$ with $i$ even.
After this last replacement, one has $C=I$.

\noindent Moreover, if none of the above two assumptions on the triviality of the values $\alpha_h(x)$ and $\alpha_h(z_j)$, with $2\leq j\leq d_2$, hold true, the same argument produces suitable matrices $A(z_1),\ldots,A(z_{d_2})$ such that the matrix $C'$ is the identity matrix. 
This concludes the analysis of case 3.

\medskip

Altogether, the assignment $w\mapsto A(x)$ for every $w\in\calX$ --- with the matrices $A(w)$'s suitably modified in case of need --- yields a homomorphism of pro-$p$ groups $\rho\colon G\to\dbU_5$ with the desired properties.
\end{proof}

We believe that the answer to the following questions is positive.

\begin{question}
\begin{itemize}
 \item[(a)]  Let $G$ be as in {\rm(1.1.a)}.
 Does $G$ satisfy the strong $n$-Massey vanishing property for every $n\geq3$?
\item[(b)]  Let $G$ be as in {\rm(1.1.b)}.
 Does $G$ satisfy the strong $n$-Massey vanishing property for every $3\leq n<p$?
\end{itemize}

\end{question}






\section{The Mina\v{c}-T\^an pro-$p$ group}

We focus now on the Mina\v{c}-T\^an pro-$p$ group
\[
 G=\langle\:x_1,\ldots,x_5\:\mid\:r=1\:\rangle\qquad\text{with }r=\left[[x_1,x_2],x_3\right][x_4,x_5].
\]
Using Proposition~\ref{prop:representations}, one may show that $G$ does not satisfy the 3-Massey vanishing property (cf. \cite[Ex.~7.2]{mt:massey}).
Our aim is to show that $G$ cannot complete into a 1-cyclotomic oriented pro-$p$ group with torsion-free orientation.


\subsection{Kummerianity and 1-cyclotomicity}\label{ssec:MT kummer}

\begin{prop}\label{prop:MT kummer}
Let $G$ be the Mina\v{c}-T\^an pro-$p$ group, and let $\theta\colon G\to 1+p\Z_p$ be a torsion-free orientation.
Then the oriented pro-$p$ group $(G,\theta)$ is Kummerian if, and only if, $x_4,x_5\in\Ker(\theta)$, and:
\begin{itemize}
 \item[(a)] $x_3\in\Ker(\theta)$; or
 \item[(b)] $x_1,x_2\in\Ker(\theta)$.
\end{itemize} 
\end{prop}

\begin{proof}
Let $c\colon G\to\Z_p(\theta)$ be an arbitrary continuous 1-cocycle, and set $c(x_i)=\lambda_i$ for $i=1,\ldots,5$.
Applying \eqref{eq:1cocycle}--\eqref{eq:1cocycle comm} one computes $c(r)=c([[x_1,x_2],x_3])+c([x_4,x_5])$, and
\begin{equation}\label{eq:c on r}
\begin{split}
c([[x_1,x_2],x_3]) &= \theta(x_1x_2)^{-1}\left(\theta(x_3)^{-1}-1\right)\left(\lambda_1(1-\theta(x_2))-\lambda_2(1-\theta(x_1))\right),\\
c([x_4,x_5]) &=\theta(x_4x_5)^{-1}\left(\lambda_4(1-\theta(x_5))-\lambda_5(1-\theta(x_4))\right).
\end{split}\end{equation}
On the other hand, $c(r)=0$ as $r=1$.

Suppose that $(G,\theta)$ is Kummerian.
Then by Lemma~\ref{lem:labute inf}, we may prescribe arbitrary values to $\lambda_1,\ldots,\lambda_5$.
If $\lambda_4=1$ and $\lambda_i=0$ for $i\neq 4$, from \eqref{eq:c on r} and from the fact that $c(r)=0$ one obtains $0=1\cdot(1-\theta(x_5))$, and thus $\theta(x_5)=1$.
Analogously, if $\lambda_5=1$ and $\lambda_i=0$ for $i\neq 5$, one deduces $\theta(x_4)=1$.
Finally, if $\lambda_4=\lambda_5=0$ from \eqref{eq:c on r} one obtains
\[ 0=c(r)=\theta(x_1x_2)^{-1}\left(\theta(x_3)^{-1}-1\right)\left(\lambda_1(1-\theta(x_2))-\lambda_2(1-\theta(x_1))\right),\]
and the arbitrariness of $\lambda_1,\lambda_2$ implies that $\theta(x_3)=1$ or $\theta(x_1)=\theta(x_2)=1$.

Conversely, suppose that $x_4,x_5\in\Ker(\theta)$, and at least one of the hypothesis (i)--(ii) holds true.
Then for any choice for $\lambda_4,\lambda_5$, by \eqref{eq:c on r} one has $c([x_4,x_5])=0$.
On the other hand, one has
\[
 c([[x_1,x_2],x_3])=\begin{cases} 0\cdot(\lambda_1(1-\theta(x_2))-\lambda_2(1-\theta(x_1)))=0 & 
 \text{if }x_3\in\Ker(\theta), \\
    \left(\theta(x_3)^{-1}-1\right)(\lambda_1\cdot0-\lambda_2\cdot0)=0 & \text{if }x_1,x_2\in\Ker(\theta).                    \end{cases}\]
Altogether, any choice for $\lambda_1,\ldots,\lambda_5$ yields a well-defined continuous 1-cocycle $c\colon G\to\Z_p(\theta)$,
and thus $(G,\theta)$ is Kummerian by Lemma~\ref{lem:labute inf}.
\end{proof}

Now consider the subgroup $H$ of $G$ generated by $x_3,x_4,x_5$ and by $y=[x_1,x_2]$.
Then $H$ is subject to the relation
$$r=[y,x_3][x_4,x_5]=1.$$
If $(G,\theta)$ is a 1-cyclotomic oriented pro-$p$ group, with $\theta$ a torsion-free orientation, then $=(H,\theta\vert_H)$ is Kummerian.
Therefore, if $c'\colon H\to\Z_p(\theta\vert_H)$ is a continuous 1-cocycle, applying \eqref{eq:1cocycle}--\eqref{eq:1cocycle comm} one obtains
\[
\begin{split}
 0=c'(r) &=c'([y,x_3])+c'([x_4,x_5])\\&=\theta(yx_3)^{-1}\left(c'(y)(1-\theta(x_3))-c'(x_3)(1-\theta(y))\right)+0\\
 &=\theta(yx_3)^{-1}c'(y)(1-\theta(x_3)),
\end{split}
\]
as $\theta(x_4)=\theta(x_5)=1$ by Proposition~\ref{prop:MT kummer}, and $y\in G'\subseteq\Ker(\theta)$.
Since $c'(y)$ may be arbitrarily chosen by Lemma~\ref{lem:labute inf}, one deduces $\theta(x_3)=1$.
This proves the following.

\begin{lem}
Let $G$ be the Mina\v{c}-T\^an pro-$p$ group, and let $\theta\colon G\to 1+p\Z_p$ be a torsion-free orientation.
If the oriented pro-$p$ group $(G,\theta)$ is 1-cyclotomic then $x_3,x_4,x_5\in\Ker(\theta)$.
\end{lem}

Moreover, if $(G,\theta)$ is 1-cyclotomic we may suppose without loss of generality that $x_2\in\Ker(\theta)$, too.
Indeed, let $v_p\colon\Z_p\to \dbN$ denote the $p$-adic valuation, and let $k\geq1$ be such that $\Img(\theta)=1+p^k\Z_p$.

Suppose first that $v_p(\theta(x_2)-1)=k$ and $v_p(\theta(x_1)-1)>k$, and set $z=x_2x_1$.
Then $\{z,x_2,x_3,x_4,x_5\}$ is a minimal generating set of $G$, $v_p(\theta(z)-1)=k$, and $G$ is subject to the relation 
\[
 \left[[z,x_2],x_3\right][x_4,x_5]=1,
\]
as $[x_2x_1,x_2]=[x_1,x_2]$.
Hence, we may assume $v_p(\theta(x_1)-1)=k$.

Consequently, there exists $\lambda\in\Z_p$ such that $\theta(x_2)=\theta(x_1)^\lambda$.
Now set $z=x_1^{-\lambda}x_2$.
Then $\{x_1,z,x_3,x_4,x_5\}$ is a minimal generating set of $G$, $\theta(z)=\theta(x_2)\theta(x_1)^{-\lambda}=1$, and $G$ is subject to the relation 
\[
 \left[[x_1,z],x_3\right][x_4,x_5]=1,
\]
as $[x_1,x_1^{-\lambda} x_2]=[x_1,x_2]$.

Therefore, from now on $\theta\colon G\to1+p\Z_p$ will denote a torsion-free orientation satisfying $x_2,\ldots,x_5\in\Ker(\theta)$.


\subsection{The subgroup $U$}\label{ssec:MT U}

Put $u=x_1^p$ and $t=x_1^{-1}x_3$.
Let $\phi\colon G\to\Z/p$ be the homomorphism defined by $\phi(x_1)=\phi(x_3)=1$ and $\phi(x_i)=0$ for $i=2,4,5$, and let $U$ be the kernel of $\phi$.
Then $U$ is a normal subgroup of $G$ of index $p$, and it is generated as a normal subgroup of $G$ by $\{u,t,x_2,x_4,x_5\}$.
In fact, $U$ is generated as a pro-$p$ group by the set
\[
 \mathcal{X}_U=\left\{\:u,\: t^{x_1^h},\:x_2^{x_1^h},\:x_4^{x_1^h},\:x_5^{x_1^h}\:\mid\:h=0,\ldots,p-1\:\right\},
\]
as $G/U=\{U,x_1U,\ldots,x_1^{p-1}U\}$.
We need to find a subset of $\calX_U$ which minimally generates $U$ as a pro-$p$ group.

\begin{prop}\label{prop:MT U}
 The set 
 \[
  \mathcal{Y}_U=\left\{\:t,\:x_2,\:x_2^{x_1},\: t^{x_1^h},\:x_4^{x_1^h},\:x_5^{x_1^h}\:\mid\:h=0,\ldots,p-1\:\right\},
 \]
is a minimal generating set of $U$ as a pro-$p$ group.
Moreover, the abelian pro-$p$ group $U^{\ab}$ is not torsion-free.
\end{prop}

\begin{proof}
The subgroup $U$ is the pro-$p$ group generated by $\calX_U$ and subject to the $p$-relations $r^{x_1^h}=1$, $h=0,\ldots,p-1$.
Since $x_3=x_1t$, one computes
\begin{equation}\label{eq:rel pastrugnata}
\begin{split}
 [[x_1,x_2],x_3] &= [x_1,x_2]^{-1}\cdot [x_1,x_2]^{x_3}\\
 &=[x_2,x_1]\cdot \left[x_1,x_2^{x_1}\right]^{t}\\
 &=x_2^{-1}\cdot x_2^{x_1}\cdot\left(\left(x_2^{x_1^2}\right)^{-1}x_2^{x_1}\right)^{t}.
\end{split}
\end{equation}
From \eqref{eq:rel pastrugnata}, and from the relation $r=1$, one deduces the equivalence
\begin{equation}\label{eq:equivalence U MT}
 \left(x_2^{x_1^2}\right)^{-1}\cdot\left(x_2^{x_1}\right)^2\cdot x_1^{-1}\equiv 1\mod U',
\end{equation}
as $[x_4,x_5]\in U'$ and $t\in U$.

Hence, $U^{\ab}$ is the abelian pro-$p$ group generated by $\calX_{U^{\ab}}=\{wU'\:\mid\:w\in\calX_U\}$ and subject to the $p$ relations induced by the equivalences $((x_2^{x_1^2})^{-1}(x_2^{x_1})^2 x_1^{-1})^{x_1^h}\equiv 1\bmod U'$, namely
\begin{equation}\label{eq:mod Up}
 \begin{split}
x_2^{x_1^2} &\equiv\left(x_2^{x_1}\right)^2 x_1^{-1}\mod U',\qquad\text{for }h=0,\\
x_2^{x_1^3} &\equiv\left(x_2^{x_1^2}\right)^2 \left(x_1^{x_2}\right)^{-1}\equiv 
\left(x_2^{x_1}\right)^3 x_1^{-2}\mod U',\qquad\text{for }h=1,\\
&\;\vdots\\
x_2^{x_1^{p}}&\equiv\left(x_2^{x_1^{p-1}}\right)^2 \left(x_1^{p-2}\right)^{-1}\equiv 
\left(x_2^{x_1}\right)^p x_1^{1-p}\mod U',\qquad\text{for }h=p-2,\\
x_2^{x_1^{p+1}}&\equiv\left(x_2^{x_1}\right)^2\cdot x_1^{-1}\equiv 
\left(x_2^{x_1}\right)^{p+1} x_1^{-p}\mod U',\qquad\text{for }h=p-1.
 \end{split}\end{equation}
On the one hand, from \eqref{eq:mod Up} one deduces that the coset $x_2^{x_1^h}U'$ is generated by $x_2U'$ and $x_2^{x_1}U'$ for every $h=2,\ldots,p-1$, so that $\calY_{U^{\ab}}=\{wU'\:\mid\:w\in\calY_U\}$ generates $U^{\ab}$ as an abelian pro-$p$ group.
On the other hand, from the equivalences with $h=p-2$ and $h=p-1$ one deduces that
\[\begin{split}
 \left(x_2^{x_1}\right)^p x_1^{1-p}\left(x_2^{u}\right)^{-1}&\equiv\left(x_2^{x_1}\right)^p x_1^{1-p-1}\equiv 
 \left(x_2^{x_1} x_1^{-1}\right)^p\equiv 1\mod U',\\   
 \left(x_2^{x_1}\right)^{p+1} x_1^{-p}\left(x_2^{ux_1}\right)^{-1}&\equiv\left(x_2^{x_1}\right)^{p+1-1} x_1^{-p}\equiv 
 \left(x_2^{x_1} x_1^{-1}\right)^p\equiv 1\mod U',   
  \end{split}
\]
as $x_2^u\equiv x_2\bmod U'$; therefore they yield equivalent relations in $U^{\ab}$.
Altogether, $U^{\ab}$ is the abelian pro-$p$ group minimally generated by $\calX_{U^{\ab}}$ and subject to the relation 
$$\left((x_2U')^{-1}\cdot x_2^{x_1}U'\right)^p=1.$$
Hence $U^{\ab}$ is not torsion-free, and $\calY_U$ is a minimal generating set of $U$ by Fact~\ref{fact:gen Gab}.
 \end{proof}

 From Proposition~\ref{prop:MT U}, one deduces that $G$ is not absolutely torsion-free, and thus the oriented pro-$p$ group $(G,\mathbf{1})$ is not 1-cyclotomic.
 

\subsection{1-cyclotomicity and the Mina\v c-T\^an pro-$p$ group}

We are ready to prove Theorem~\ref{thm:MinacTan G}.

\begin{proof}[Proof of Theorem~\ref{thm:MinacTan G}]
Suppose for contradiction that there exists a torsion free orientation $\theta\colon G\to1+p\Z_p$ such that the oriented pro-$p$ group $(G,\theta)$ is 1-cyclotomic.
Then by \S~\ref{ssec:MT kummer}, we may assume without loss of generality that $x_2,\ldots,x_5\in\Ker(\theta)$, while $\theta(x_1)\neq1$ by \S~\ref{ssec:MT U}.
Set $\lambda\in p\Z_p\smallsetminus\{0\}$ such that $\theta(x_1)=1+\lambda$.

Consider the oriented pro-$p$ group $(U,\theta\vert_U)$, and set $K=K_{\theta\vert_U}(U)$, $\bar U=U/K$.
Our goal is to show that the oriented pro-$p$ group $(\bar U,(\theta\vert_U)_{/K})$ is not $(\theta\vert_U)_{/K}$-abelian, so that $(U,\theta\vert_U)$ is not Kummerian by Proposition~\ref{prop:kummer thetabelian}, and thus $(G,\theta)$ is not 1-cyclotomic.

Since $K\subseteq \Phi(U)$, by Proposition~\ref{prop:MT U} the set $\calY_{\bar U}=\{wK\:\mid\:w\in\calY_U\}$ is a minimal generating set of $\bar U$.
Now, since $\theta(t)=\theta(x_1)=(1+\lambda)^{-1}$, one has $w^t\equiv w^{1+\lambda}\bmod K$ for every $w\in U$.
Therefore, from \eqref{eq:rel pastrugnata}, and from the fact that $[x_4,x_5]\in\Ker(\theta\vert_U)'\subseteq K$, one obtains
\[
[x_1,x_2]^{-1}\left([x_1,x_2]^{x_1}\right)^t\equiv[x_1,x_2]^{-1}\left([x_1,x_2]^{x_1}\right)^{(1+\lambda)^{-1}}\equiv 1\mod K,
\]
and consequently
\begin{equation}\label{eq:rel mod K MT}
\begin{split}
 [x_1,x_2]^{x_1} &\equiv[x_1,x_2]^{1+\lambda}\mod K,\\
 [x_1,x_2]^{x_1^2}&\equiv [x_1,x_2]^{(1+\lambda)^2}\mod K,\\
 &\;\vdots\\
 [x_1,x_2]^{x_1^{p-1}}&\equiv[x_1,x_2]^{(1+\lambda)^{p-1}}.
\end{split}
\end{equation}
Set  
\[\mu=(1+\lambda)^0+(1+\lambda)^1+\ldots+(1+\lambda)^{p-1}=\dfrac{(1+\lambda)^p-1}{\lambda}.\]
Then $\mu\neq0$ (as $\lambda\neq0$), and $p\mid\mu$.
Since $[x_1,x_2]=(x_2^{x_1})^{-1}x_2$, replacing the coset $x_2^{x_1}K$ with the coset $[x_1,x_2]K$ in $\calY_{\bar U}$ yields another minimal generating set --- let us call it $\calY'_{\bar U}$ --- of $\bar U$.
Now, from \eqref{eq:rel mod K MT} one obtains
\[
 \begin{split}
  [u,x_2]&=[x_1,x_2]^{x_1^{p-1}}\cdots[x_1,x_2]^{x_1}\cdot [x_1,x_2]\\
  &\equiv [x_1,x_2]^{(1+\lambda)^{p-1}}\cdots[x_1,x_2]^{1+\lambda}\cdot[x_1,x_2]\mod K\\
  &\equiv [x_1,x_2]^{\mu}\mod K
 \end{split}
\]
--- observe that $[x_1,x_2]^{x_i^h}\in\Ker(\theta\vert_U)$ for every $h$, and thus all such elements commute modulo $K$.
Therefore, one has the relation
\[
 \left([x_1,x_2]K\right)^\mu=[uK,x_2K]
\]
between elements of the minimal generating set $\calY'_{\bar U}$, and by \cite[Thm.~8.1]{eq:kummer} this relation prevents the oriented pro-$p$ group $(\bar U,(\theta\vert_U)_{/K})$ from being Kummerian --- and thus also $(\theta\vert_U)_{/K}$-abelian.
\end{proof}

From Theorem~\ref{thm:MinacTan G} we obtain a new family of pro-$p$ groups which cannot complete into 1-cyclotomic oriented pro-$p$ groups.

\begin{cor}
 Let $G$ be the pro-$p$ group with presentation
 \[G=\left\langle x_1,\ldots,x_n,x_{n+1},x_{n+2} \:\mid 
 \: \left[\left[\ldots[[x_1,x_2],x_3],\ldots x_{n-1}\right],x_n\right][x_{n+1},x_{n+2}]=1\right\rangle,
\]
with $n\geq3$.
Then $G$ cannot complete into a 1-cyclotomic oriented pro-$p$ group with torsion-free orientation.
\end{cor}

\begin{proof}
 Set $y=[\ldots[x_1,x_2],\ldots x_{n-2}]$, and let $H$ be the subgroup of $G$ generated by $\{y,x_{n-1},\ldots,x_{n+2}\}$.
 Then 
 \[
  H=\langle\:y,x_{n-1},\ldots,x_{n+2}\:\mid\:[[y,x_{n-1}],x_n][x_{n+1},x_{n+2}]\:\rangle
 \]
is isomorphic to the Mina\v c-T\^an pro-$p$ group, and hence it cannot complete into a 1-cyclotomic oriented pro-$p$ group with torsion-free orientation by Theorem~\ref{thm:MinacTan G}.
\end{proof}
The following question remains open (cf. \cite[Ex.~3.2]{BCQ}).

\begin{question}
Is the Mina\v c-T\^an pro-$p$ group $G$ a Bloch-Kato pro-$p$ group? Namely, is the $\Z/p\Z$-cohomology algebra of every closed subgroup of $G$ a quadratic algebra?
\end{question}

\section*{Declarations}
\subsection*{Ethical Approval} Not applicable. 

\subsection*{Competing interests} The author has no competing interests to declare that are relevant to the content of this article.
 
\subsection*{Authors' contributions} Not applicable.

\subsection*{Funding} The author was partially supported by the ``Giovani Talenti'' Prize (2017), funded by the University of Milano-Bicocca (Italy) and sponsored by the Accademia Nazionale dei Lincei.

\subsection*{Availability of data and materials} Not applicable.



\end{document}